\documentclass[12pt]{article}
\usepackage[margin=1in]{geometry}
\usepackage[utf8]{inputenc}

\usepackage[english]{babel}
\usepackage[sort,numbers]{natbib}
\usepackage[dvipsnames]{xcolor}

\usepackage{amsmath, amssymb, amsthm,mathtools,dsfont}
\usepackage{bbm}
\usepackage{blkarray}
\usepackage{cancel}
\usepackage{natbib}
\usepackage{enumitem}
\usepackage{euscript}
\usepackage{float}
\usepackage{bm}
\usepackage{graphicx}
\usepackage{mathrsfs}
\usepackage{microtype}
\usepackage{ragged2e}
\usepackage{sectsty}
\usepackage{thmtools}
\usepackage{tikz}
\usepackage[Symbolsmallscale]{upgreek}

\usepackage{microtype}
\usepackage{graphicx}
\usepackage{subfigure}
\usepackage{hyperref}

\hypersetup{citecolor=red, colorlinks=true, linkcolor=blue}

\newcommand{\cA}{\mathcal{A}}

\newcommand{\cC}{\mathcal{C}}

\newcommand{\cE}{\mathcal{E}}
\newcommand{\cF}{\mathcal{F}}

\newcommand{\cH}{\mathcal{H}}

\newcommand{\cO}{\mathcal{O}}
\newcommand{\cP}{\mathcal{P}}
\newcommand{\cQ}{\mathcal{Q}}
\newcommand{\cR}{\mathcal{R}}

\newcommand{\cT}{\mathcal{T}}

\newcommand{\cX}{\mathcal{X}}

\newcommand{\NN}{\mathbb{N}}

\newcommand{\PP}{\mathbb{P}}

\newcommand{\RR}{\mathbb{R}}

\newcommand{\1}{\mathds{1}}

\newcommand{\kl}[2]{\text{KL}(#1 \| #2)}
\newcommand*{\tv}[2]{\mathrm{d_{TV}}(#1, #2)}
\newcommand*{\hel}[2]{\mathrm{d_{H^2}}(#1, #2)}
\newcommand*{\chis}[2]{\chi^2(#1\| #2)}

\newcommand*{\triplenorm}[1]{{\left\vert\kern-0.25ex\left\vert\kern-0.25ex\left\vert #1
    \right\vert\kern-0.25ex\right\vert\kern-0.25ex\right\vert}}

\DeclareMathOperator{\supp}{supp}

\newcommand{\R}{\mathbb{R}}
\newcommand{\Rd}{\mathbb{R}^d}
\renewcommand{\phi}{\varphi}
\newcommand{\eps}{\varepsilon}
\newcommand{\sse}{\subseteq}

\newcommand*{\E}{\mathbb E}
\newcommand*{\p}[1]{\mathbb P\left\{#1\right\}}

\newcommand*{\ep}{\varepsilon}
\newcommand*{\defeq}{\coloneqq}
\newcommand*{\rd}{\mathrm{d}}
\newcommand*{\dd}{\, \rd}
\DeclareMathOperator*{\argmin}{argmin}
\DeclareMathOperator*{\argmax}{argmax}

\newcommand{\Tnn}{\hat{T}_{\text{1NN}}}

\DeclareMathOperator{\OT}{S_0}
\DeclareMathOperator{\OTep}{S_\eps}
\usepackage{fancyhdr}

\usepackage[capitalize,noabbrev]{cleveref}

\theoremstyle{plain}
\newtheorem{theorem}{Theorem}[section]
\newtheorem{proposition}[theorem]{Proposition}
\newtheorem{lemma}[theorem]{Lemma}
\newtheorem{corollary}[theorem]{Corollary}
\theoremstyle{definition}

\theoremstyle{remark}
\newtheorem{remark}[theorem]{Remark}
\newtheorem{example}[theorem]{Example}

\begin{document}

\vspace*{0.3in}

\begin{center} {\LARGE{{Minimax estimation of discontinuous optimal transport maps: The semi-discrete case}}}

{\large{
\vspace*{.3in}
\begin{tabular}{cccc}
Aram-Alexandre Pooladian$^{1,*}$, Vincent Divol$^{2,*}$, Jonathan Niles-Weed$^{1,3}$\\
\end{tabular}
{
\vspace*{.1in}
\begin{tabular}{c}
				$^1$Center for Data Science, New York University\\
				$^2$\textsc{Ceremade}, Universit{\'e} Paris Dauphine-PSL \\
                $^3$Courant Institute of Mathematical Sciences, New York University\\
 				\small{\texttt{aram-alexandre.pooladian@nyu.edu,vincent.divol@psl.eu, jnw@cims.nyu.edu}}\\
\end{tabular} 
}

}}
\vspace*{.1in}

\today

\end{center}

\vspace*{.1in}

\begin{abstract}
We consider the problem of estimating the optimal transport map between two probability distributions, $P$ and $Q$ in $\Rd$, on the basis of i.i.d.\ samples.
All existing statistical analyses of this problem require the assumption that the transport map is Lipschitz, a strong requirement that, in particular, excludes any examples where the transport map is discontinuous.
As a first step towards developing estimation procedures for discontinuous maps, we consider the important special case where the data distribution $Q$ is a discrete measure supported on a finite number of points in $\Rd$.
We study a computationally efficient estimator initially proposed by \citep{pooladian2021entropic}, based on entropic optimal transport, and show in the semi-discrete setting that it converges at the minimax-optimal rate $n^{-1/2}$, independent of dimension.
Other standard map estimation techniques both lack finite-sample guarantees in this setting and provably suffer from the curse of dimensionality.
We confirm these results in numerical experiments, and provide experiments for other settings, not covered by our theory, which indicate that the entropic estimator is a promising methodology for other discontinuous transport map estimation problems.
\end{abstract}

\footnotetext{*Pooladian and Divol contributed equally to this work.}
\section{Introduction}\label{sec: intro}

The theory of optimal transport (OT) defines a natural geometry on the space of probability measures \cite{San15,Vil08} and has become ubiquitous in modern data-driven tasks. In this area, \textit{optimal transport maps} are a central object of study: suppose $P$ and $Q$ are two probability distributions with finite second moments, with $P$ having a density with respect to the Lebesegue measure on $\Rd$. Then, Brenier's theorem (see \cref{sec: background_ot}) states that there exists a convex function $\phi_0$ whose gradient defines a unique \textit{optimal transport map} between $P$ and $Q$. This map is optimal in the sense that it minimizes the following objective function:
\begin{align}\label{eq: monge_intro}
    \nabla\phi_0 \defeq \argmin_{T \in \cT(P,Q)} \int \tfrac12 \|x - T(x)\|^2 \dd P(x)\,,
\end{align}
where $\cT(P,Q) \defeq \{ T : \Rd \to \Rd  \ | \ X \sim P,\ T(X)\sim Q\}$ is the set of transport maps between $P$ and $Q$. The optimal value of the objective function in \Cref{eq: monge_intro} is called the (squared) 2-Wasserstein distance, written explicitly as 
\begin{align*}
\OT(P,Q) = \int \tfrac12 \|x - \nabla \phi_0(x)\|^2 \dd P(x)\,,
\end{align*}
though a more general formulation is available (see \Cref{sec: background_ot}). Computing or approximating $\OT(P,Q)$ as well as $\nabla\phi_0$ has found use in several academic communities, such as economics ~\cite{carlier2016vector,chernozhukov2017monge,torous2021optimal,gunsilius2021matching}, computational biology ~\cite{bunne2021learning,bunne2022supervised,lubeck2022neural,schiebinger2019optimal,moriel2021novosparc,Demetci2021.SCOTv2,dai2018autoencoder}, and computer vision ~\cite{SolPeyKim16,SolGoePey15,feydy2017optimal}, among many others.
 
Practitioners seldom have access to $P$ or $Q$, but instead have access to i.i.d.~samples $X_1,\ldots,X_n \sim P$ and $Y_1,\ldots,Y_n \sim Q$. On the basis of these samples, practitioners face both computational and statistical challenges when estimating $\nabla \phi_0$.
From a theoretical perspective, the statistical task of estimating optimal transport maps has attracted much interest in the last few years \cite{hutter2021minimax,muzellec2021near,manole2021plugin,deb2021rates,pooladian2021entropic, divol2022optimal,GhoSen22}.

The first finite-sample analysis of this problem was performed by \citep{hutter2021minimax}, who proposed an estimator for $\nabla \phi_0$ under the assumption that $\phi_0$ is $s+1$-times continuously differentiable, for $s > 1$.
They showed that a wavelet-based estimator $\hat{\phi}_{\text{W}}$ satisfies
\begin{align*}
	\E \| \nabla \hat{\phi}_{\text{W}} - \nabla \phi_0\|_{L^2(P)}^2 \lesssim  n^{-\frac{2s}{2s + d - 2}}\log^2(n)\,,
\end{align*}
and that this rate is minimax optimal up to logarithmic factors.
Their analysis requires that $P$ and $Q$ have bounded densities with compact support $\Omega \sse \Rd$, and that $\phi_0$ be both strongly convex and smooth.
Implementing the estimator $\hat{\phi}_{\text{W}}$ is computationally challenging even in moderate dimensions, and is practically infeasible for $d > 3$.
Follow up work has proposed alternative estimators which improve upon $\hat{\phi}_{\text{W}}$ either in computational efficiency or in the generality in which they apply.
Though these subsequent works go significantly beyond the setting considered by \citep{hutter2021minimax}, none has eliminated the crucial assumption that $\phi_0$ is smooth, i.e., that the transport map $\nabla \phi_0$ is Lipschitz.

We highlight two estimators proposed in this line of work that are particularly practical.
\citep{manole2021plugin} study the $1$-Nearest Neighbor estimator $\hat{T}_{\text{1NN}}$. 
This estimator is obtained by solving the empirical optimal transport problem between the samples,
which is then extended to a function defined on $\R^d$ using a projection scheme; see \cref{sec: num_all} for more details. Given $n$ samples from the source and target measures in $\Rd$, $\hat{T}_{\text{1NN}}$ has a runtime of $\cO(n^3)$ via the Hungarian Algorithm  \citep[see][Chapter 3]{PeyCut19}, and, for $d \geq 5$, achieves the rate 
\begin{equation}\label{1nn_rate}
	\E \| \hat T_{\text{1NN}} - \nabla \phi_0\|_{L^2(P)}^2 \lesssim n^{-\frac{2}{ d }}\,
\end{equation}
whenever the optimal Brenier potential $\phi_0$ is smooth and strongly convex, and under mild regularity conditions on $P$.
In another work, \citep{pooladian2021entropic} conducted a statistical analysis of an estimator originally proposed by \citep{seguy2017large} based on entropic optimal transport.
The efficiency of Sinkhorn's algorithm for large-scale problems~\cite{cuturi2013sinkhorn,PeyCut19} makes this estimator attractive from a computational perspective, and \citep{pooladian2021entropic} also give statistical guarantees, though these fall short of being minimax-optimal.

Despite this progress, none of the aforementioned results can be applied in situations where $\nabla \phi_0$ is not Lipschitz.
And in practice, even requiring the \emph{continuity} of the transport map can be far too stringent.
It is indeed too much to hope for that an underlying data distribution (e.g.~over the space of images) has one single connected component; this is supported by recent work that stipulates that the underlying data distribution is the union of \textit{disjoint} manifolds of varying intrinsic dimension \cite{brown2022union}. In such a setting, the transport map $\nabla\phi_0$ will not be continuous, demonstrating the need of considering the problem of the statistical estimation of \emph{discontinuous} transport maps to get closer to real-world situations.

As a first step, we choose to focus on the case where the target distribution $Q=\sum_{j=1}^J q_j \delta_{y_j}$ is discrete while the source measure $P$ has full support, often called the \emph{semi-discrete} setting in the optimal transport literature.
 In this setting, the optimal transport map $\nabla \phi_0$ is constant over regions known as Laguerre cells (each cell corresponding to a different atom of the discrete measure), while displaying discontinuities on their boundaries (see \cref{sec: semidiscrete_background} for more details). \cref{fig: laguerre_intro} provides such an example.
Semi-discrete optimal transport therefore provides a natural class of discontinuous transport maps.

\begin{figure}[ht]
	\begin{center}
		\centerline{\includegraphics[width=0.5\columnwidth]{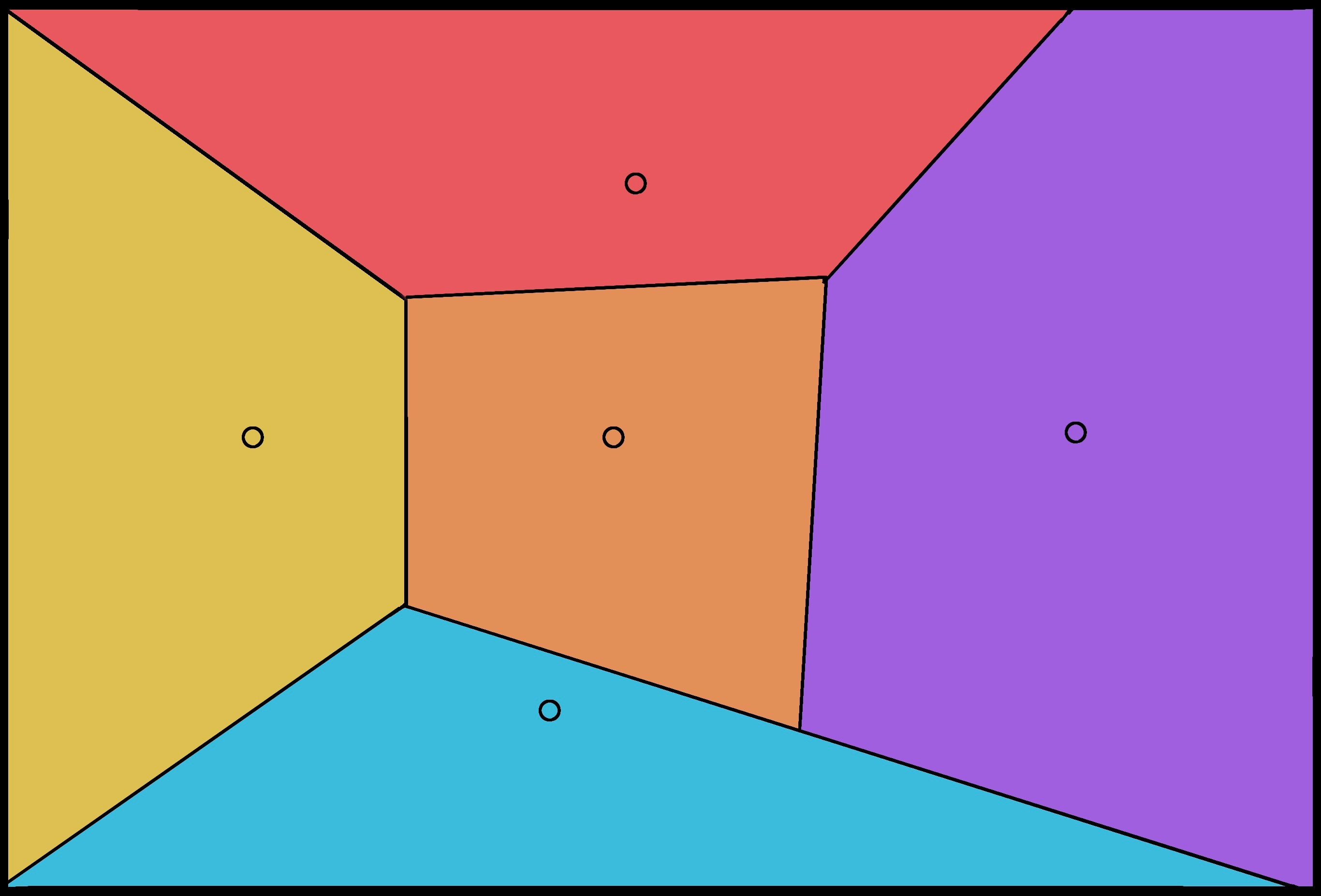}}
		\caption{An illustration of a semi-discrete optimal transport map. The support of $P$, the whole rectangle, is partitioned into regions, each of which is transported to one of the atoms of the discrete target measure $Q$. The resulting map is discontinuous at the boundaries of each cell.}
		\label{fig: laguerre_intro}
	\end{center}
\end{figure}
We focus on this setting for two reasons.
First, it has garnered a lot of attention in recent years, in both computational and theoretical circles~\citep[see, e.g.,][]{merigot_santambrogio_sarrazin, altschuler2021asymptotics,chen2022semidiscrete}, due in particular to its connection with the quantization problem \cite{graf2007foundations}.
Second, the semi-discrete setting is intriguing from a statistical perspective: existing results show that statistical estimation problems involving semi-discrete optimal transport can escape the curse of dimensionality~\cite{ForHutNit19,BarLou19,BarGonLou22,HunStaMun22}.
For example, \citep[Theorem 3.2]{HunStaMun22} show that if $P_n$ and $Q_n$ are empirical measures consisting of i.i.d.\ samples from $P$ and $Q$, then the semi-discrete assumption implies
\begin{equation*}
	\E |S_0(P, Q) - S_0(P_n, Q_n)| \lesssim n^{-1/2}\,.
\end{equation*}
These results offer the tantalizing possibility that semi-discrete transport maps can be estimated at the rate $n^{-1/2}$, in sharp contrast to the dimension-dependent rates obtained in bounds such as~\eqref{1nn_rate}.
However, the optimal rates of estimation for semi-discrete transport maps are not known, and no estimators with finite-sample convergence guarantees exist.

\subsubsection*{\textbf{Main Contributions}}
We show that the computationally efficient estimator $\hat{T}_{\eps}$ based on entropically regularized optimal transport, originally studied in \cite{seguy2017large, pooladian2021entropic}, provably estimates discontinuous semi-discrete optimal transport maps at the optimal rate.
More precisely, our contributions are the following: 
\begin{enumerate}
    \item For $Q$ discrete and $P$ with full support on a compact, convex set, we show that $\hat{T}_\eps$ achieves the following \textit{dimension-independent} convergence rate to the optimal transport map (see \cref{thm: mapest_semidiscrete})
\begin{equation}\label{eq:final_rate_entropic}
     \E \| \hat{T}_{\eps} - \nabla \phi_0\|^2_{L^2(P)} \lesssim n^{-1/2}\,,
\end{equation}
when the regularization parameter $\eps  \asymp n^{-1/2}$.
We further show (\cref{prop:minimax}) that this rate is minimax optimal.
\item As a by-product of our analysis, we give new \textit{parametric} rates of convergence to the entropic Brenier map $T_\eps$, a result which improves exponentially on prior work in the dependence on $\eps$ (see \cref{thm: entmap_semidiscrete_samp_full} and \cref{remark: rigollet2022}).
\item Our proof technique requires several new results, including a novel stability bound for the entropic Brenier maps (\cref{prop: entmap_stability}), and a new stability result for the entropic dual Brenier potentials in the semi-discrete case (\cref{prop:stability_potentials}).
\item
We show that, unlike $\hat T_\ep$, the 1-Nearest-Neighbor estimator is provably suboptimal in the semi-discrete setting (see \cref{prop:1NN_suboptimal}) by exhibiting a discrete measure $Q$ such that the risk suffers from the curse of dimensionality:
\begin{align*}
    \E \|\hat{T}_{\text{1NN}} - \nabla \phi_0\|^2_{L^2(P)} \gtrsim n^{-1/d}\,.
\end{align*}
\item In \cref{sec: num_all}, we verify our theoretical findings on synthetic experiments. We also show by simulation that the entropic estimator appears to perform well even outside the semi-discrete setting, suggesting it as a promising choice for estimating other types of discontinuous maps.
\end{enumerate}
\subsubsection*{Notation}
The Euclidean ball centered at $a$ with radius $r > 0$ is written as $B(a;r)$. The symbols $C$  and $c$ denote positive constants whose value may change from line to line. Write $a\lesssim b$ and $a \asymp b$ if there exist constants $c, C > 0$ such that $a \leq Cb$ and $cb \leq a \leq Cb$, respectively. For an integer $N \in \NN$, we let $[N] \defeq \{1,\ldots, N\}$. For a function $f$ and a probability measure $\rho$, we write $\|f\|_{L^2(\rho)}^2 \defeq \E_{X \sim \rho}\|f(X)\|^2\,.$ Similarly, we write $\text{Var}_\rho(f) \defeq \E_{X\sim \rho}[(f(X) - \E_{X \sim \rho}[f(X)])^2]$ for the variance of $f$ with respect to $\rho$.

\section{Background on optimal transport}\label{sec: background_all}
\subsection{Optimal transport}\label{sec: background_ot}
We define $\cP(\Omega)$ to be the space of probability measures  whose support lies in a compact subset $\Omega \sse \Rd$. If a probability measure $P$ has a density with respect to the Lebesgue measure on $\Rd$ with support $\Omega \sse \Rd$, then we write $P \in \cP_{\text{ac}}(\Omega)$.  

For two probability measures $P,Q \in \cP(\Omega)$, we define the \textit{(squared) $2$-Wasserstein distance} to be \cite{Kan42}
\begin{align}\label{eq: kant_ot_p}
\OT(P,Q) := \min_{\pi \in \Gamma(P,Q)} \iint \tfrac12 \|x - y\|^2 \dd \pi(x,y)\,,
\end{align}
where $\pi \in \Gamma(P,Q) \sse \cP(\Omega \times \Omega)$ such that for any event $A$,
\begin{align*}
\pi(A \times \Omega) = P(A)\,,\quad \pi( \Omega \times A) = Q(A) \,.
\end{align*}
We call $\Gamma(P,Q)$ the set of \textit{couplings} between $P$ and $Q$. In this work, we focus on the squared-Euclidean cost but \cref{eq: kant_ot_p} is well-defined for convex, lower-semicontinuous costs; see \cite{Vil08, San15} for more information on optimal transport under general costs.

\Cref{eq: kant_ot_p} is a convex optimization problem on the space of joint measures, and a minimizer, denoted $\pi_0$, always exists; we call $\pi_0$ an \textit{optimal plan} from $P$ to $Q$. Moreover, \cref{eq: kant_ot_p} possesses the following {dual formulation},
\begin{align}\label{eq: kant_ot_d}
\begin{split}
    \OT(P,Q) &= \tfrac12 M_2(P) + \tfrac12 M_2(Q) -\inf_{(\phi,\psi) \in \Phi} \int \phi \dd P + \int \psi \dd Q 
\end{split}
\end{align}
where $M_2(P) :=  \int \|x\|^2 \dd P(x)$ (similarly for $M_2(Q)$) and the functions $(\phi,\psi) \in \Phi \sse L_1(P) \times L_1(Q)$ satisfy
\begin{align*}
\langle x , y \rangle \leq \phi(x) + \psi(y) \ \text{for all } x,y \in \Omega\,,
\end{align*}
As with the primal formulation, the infimum in \cref{eq: kant_ot_d}  is attained at functions $(\phi_0,\psi_0)$. These minimizers are called \textit{(optimal) Brenier potentials}. In particular, at optimality, we have that these Brenier potentials are convex conjugates of one another, i.e. the Legendre transform of one of the potentials gives the other:
\begin{align}\label{eq: legendre_transform}
\phi_0^*(y) := \sup_x \{ \langle x , y \rangle - \phi_0(x) \} = \psi_0(y)\,,
\end{align}
and vice-versa.

Apart from these two formulations of optimal transport under the squared-Euclidean cost, there exists a third, known as the Monge problem: 
\begin{align}\label{eq: monge_eq}
T_0 := \argmin_{T \in \cT(P,Q)} \int \tfrac12 \|x - T(x)\|^2 \dd P(x)\,,
\end{align}
where $\cT(P,Q)$ is the set of admissible transport maps, i.e. for $X \sim P$, $T(X) \sim Q$. This optimization problem is non-convex in $T$, and a solution is not always guaranteed to exist for arbitrary $P$ and $Q$. 

The following theorem unifies these three formulations of optimal transport under the squared-Euclidean cost:
\begin{theorem}[Brenier's theorem; \citealp{brenier1991polar}] \label{thm: brenier}
Let $P \in \cP_{\text{ac}}(\Omega)$ and let $Q \in \cP(\Omega)$, then
\begin{enumerate}
\item the solution to \cref{eq: monge_eq} exists and is of the form $T_0 = \nabla \phi_0$, where $\phi_0$ solves \cref{eq: kant_ot_d}
\item $\pi_0$ is also uniquely defined as  $$\dd \pi_0(x,y) = \dd P(x) \delta_{\{\nabla \phi_0(x)\}}(y)\,.$$
\end{enumerate}
\end{theorem}
When we want to place emphasis on the underlying measures, we will write $\phi_0 = \phi_0^{P\to Q}$, $\psi_0 = \psi_0^{P\to Q}$ and $T_0 = T_0^{P\to Q}$.

\subsubsection{OT in the semi-discrete case}\label{sec: semidiscrete_background}
In optimal transport, the semi-discrete setting refers to the case where $P$ has as density with respect to the Lebesgue measure on $\Rd$, and $Q$ is a discrete measure supported on points. The following theorem characterizes the optimal transport map in this situation, which exhibits a particular structure compared to the general results in the previous section. Let $[J]=\{1,\dots,J\}$.
\begin{proposition}[\citealp{aurenhammer1998minkowski}] If $P\in \cP_{\text{ac}}(\Omega)$ and $Q$ is a discrete measure supported on the points $y_1,\dots,y_J$, then the optimal transport map $\nabla\phi_0$ is given by
\begin{align}
\nabla \phi_0(x) \defeq \argmax_{j \in [J]}\{ \langle x , y_j \rangle - \psi_0(y_j) \}\,,
\end{align}
where $\psi_0$ is the dual to $\phi_0$ in the sense of \cref{eq: legendre_transform}.    
\end{proposition}
Here, the optimal dual Brenier potential $\psi_0$ can be identified with a \textit{vector} in $\R^J$, defined by the number of atoms, and the optimal Brenier potential is consequently given by $$\phi_0 \defeq \max_{j \in [J]}\{\langle x , y_j \rangle - \psi_0(y_j) \}\,.$$
Although $\phi_0$ is not differentiable, only subdifferentiable, we still use the gradient notation as $\nabla \phi_0$ is well-defined $P$-almost everywhere. 

The map $\nabla \phi_0$ partitions the space into $J$ convex polytopes $L_j \defeq \nabla\phi_0^{-1}(\{y_j\})$ called \emph{Laguerre cells}; recall \Cref{fig: laguerre_intro}. From this definition, it is clear that for a given $x \in L_j$, $x \mapsto \nabla\phi_0(x) = y_j$ is the optimal transport mapping. The difficulty in finding this map lies in determining the cells $L_j$, or equivalently the dual variables $\psi_0(y_j)$.

\subsection{Entropic optimal transport} \label{sec: background_eot}
Entropic regularization was introduced to both optimal transport and machine learning communities in the seminal paper by \citep{cuturi2013sinkhorn}, allowing approximate optimal transport distances to be computed at unprecedented speeds. 
Entropic optimal transport (EOT) is defined as the following regularized version of \cref{eq: kant_ot_p}: for $\eps > 0$
\begin{align}\label{eq: kant_eot_p}
\begin{split}
    \OTep(P,Q) := \min_{\pi \in \Gamma(P,Q)} &\iint \tfrac12 \|x-y\|^2 \dd \pi(x,y) + \eps \kl{\pi}{P\otimes Q}\,,
\end{split}
\end{align}
where $\kl{\mu}{\nu} = \int \log \frac{\dd \mu}{\dd \nu} \dd \mu$ when $\mu\in \cP(\Omega)$ is absolutely continuous with respect to $\nu\in \cP(\Omega)$. 
This speedup is due to the elegant connection of \eqref{eq: kant_eot_p} to Sinkhorn's algorithm; we refer the interested reader to \citep[Chapter 4]{PeyCut19} for more information. 
The computational tractability of $\OTep$ compared to $\OT$ when dealing with many samples lends itself to being a central object of study in its own right~\citep[see, e.g.,][]{genevay2019sample, mena2019statistical, chizat2020faster, rigollet2022sample,gonzalez2022weak}.

\Cref{eq: kant_eot_p} admits the following dual formulation, which is now an unconstrained optimization problem \cite{genevay2019entropy,marino2020optimal}
\begin{equation}\label{eq: kant_eot_d}
\begin{split}
\OTep(P,Q) =\tfrac12 M_2(P) + \tfrac12 M_2(Q) - &\inf_{ \phi,\psi } \bigg(\int \phi \dd P + \int \psi \dd Q \\
    &\quad + \eps \hspace{-.05cm} \iint (e^{(\langle x, y \rangle - \phi(x) - \psi(y))/\eps }-1) \dd P(x) \dd Q(y)\bigg),
    \end{split}
\end{equation}
where $(\phi,\psi) \in L_1(P) \times L_1(Q)$. When $P$ and $Q$ have finite second moments, \cref{eq: kant_eot_p} admits a \textit{unique} minimizer, $\pi_\eps$ and we have the existence of minimizers to \cref{eq: kant_eot_d}, which we denote as $(\phi_\eps,\psi_\eps)$. We call $\pi_\eps$ the \textit{entropic optimal plan} and $(\phi_\eps,\psi_\eps)$ are called \textit{entropic Brenier potentials}. The following optimality relation further relates these primal and dual solutions \cite{Csi75}:
\begin{align*}
\dd\pi_\eps(x,y) := e^{(\langle x , y \rangle - \phi_\eps(x) - \psi_\eps(y))/\eps} \dd P(x) \dd Q(y)\,.
\end{align*}
As a consequence, the following relationship holds at optimality:
\begin{align*}
\OTep(P,Q) =  \tfrac12 M_2(P) + \tfrac12 M_2(Q)-\hspace{-.1cm}\int \phi_\eps \dd P - \hspace{-.1cm} \int \psi_\eps \dd Q\,, 
\end{align*}
and, moreover, we can define versions of $\phi_\eps$ and $\psi_\eps$ such that the following relationships hold \citep[see][]{mena2019statistical, nutz2021entropic} over all $x \in \Rd$ and $y \in \Rd$, respectively:
\begin{align}
& \phi_\eps(x) = \eps \log \int e^{(\langle x , y \rangle - \psi_\eps(y))/\eps} \dd Q(y)\,, \label{eq: opt_cond_1} \\
& \psi_\eps(y) = \eps \log \int e^{(\langle x , y \rangle - \phi_\eps(x))/\eps} \dd P(x)\,, \label{eq: opt_cond_2}
\end{align}
which are smoothed version of the Legendre transform, see \Cref{sec:reminder} for details.
In what follows, we always assume that we have selected $\phi_\eps$ and $\psi_\eps$ so that these identities hold.

\subsubsection{Entropic Brenier Map}\label{sec:entropic_brenier_map}
If $(X,Y)\sim\pi_\eps$, we may define the conditional probability $\pi^x_\eps$ of $Y$ given that $X=x$, with density
\begin{equation}\label{eq:cond_density}
    \frac{\dd \pi^x_\eps}{\dd Q}(y) \propto \exp\left((\langle x,y\rangle -\psi_\eps(y))/\eps\right)\,.
\end{equation}
The barycentric projection of the optimal entropic coupling $\pi_\eps$, or \textit{entropic Brenier map}, is a central object of study in several works e.g. \cite{goldfeld2022limit,pooladian2021entropic,del2022improved,rigollet2022sample}, defined as
\begin{align}
T_\eps(x)  =\hspace{-.05cm}  \int \hspace{-.05cm} y \dd \pi^x_\eps(y) &=  \nabla \phi_\eps(x)\,, \label{eq: map_def}
\end{align}
where $\pi_\eps^x$ is as in \cref{eq:cond_density}. Note that this quantity is well defined for all $x \in \Rd$ as long as the source and target measures have compact support; in particular, it applies to both discrete and continuous measures.
The second equality follows from \cref{eq: opt_cond_1} and the dominated convergence theorem.
As in the unregularized case, we will write  $\phi_\eps = \phi_\eps^{P\to Q}$, $\psi_\eps = \psi_\eps^{P\to Q}$ and $T_\eps = T_\eps^{P\to Q}$ when we want to emphasize on the  dependency with respect to the underlying measures.

This particular barycentric projection was proposed as a tool for large-scale optimal transport by \citep{seguy2017large}, but analyzed statistically for the first time by \citep{pooladian2021entropic} as an estimator for the optimal transport map. We mention some of their results to highlight the differences with our new results for the semi-discrete setting in \Cref{sec:main_results}. First, they prove the following approximation result for $T_\eps$.

\begin{proposition}[{\citealp[Corollary 1]{pooladian2021entropic}}]\label{thm: pnw_main_thm}
    Let $P,Q$ be compactly supported absolutely continuous measures on a compact set $\Omega \sse \Rd$ with densities $p$ and $q$, that are bounded away from $0$ and $\infty$. Assume that $\phi_0$ is smooth and strongly convex, and that $\phi_0^*$ is at least $\cC^3$. Then,
    \begin{align}\label{eq: pnw_main_thm}
        \| T_\eps - \nabla \phi_0 \|^2_{L^2(P)} \lesssim \eps^2 \,.
    \end{align}
\end{proposition}
Their main statistical result is the following theorem:
\begin{proposition}[{\citealp[Theorem 3]{pooladian2021entropic}}]\label{thm: pnw_stat_thm}
    Suppose the same assumptions as \cref{thm: pnw_main_thm}, and let $P_n$ and $Q_n$ denote the empirical measures of $P$ and $Q$ constructed from i.i.d.~samples. Let $\hat{T}_\eps = T_\eps^{P_n \to Q_n}$ denote the entropic Brenier map from $P_n$ to $Q_n$ and let $T_0 = \nabla \phi_0$ be the optimal transport map from $P$ to $Q$. Then, if $\eps \asymp n^{-\frac{1}{d' + 3}}$
    \begin{align}\label{eq: pnw_stat_thm}
       \E \| \hat{T}_{\eps} - T_0 \|^2_{L^2(P)} \lesssim n^{-\frac{3}{2(d' + 3)}}\log(n)\,,
    \end{align}
    where $d' = 2 \lceil d/2 \rceil$.
\end{proposition}
Note that in particular the the rate of convergence of the entropic estimator critically depends on the ambient dimension $d$ in the continuous-to-continuous case.

\subsubsection{Related work}
Characterizing the convergence of entropic objects (e.g. potentials, cost, plans) to their unregularized counterparts in the $\eps \to 0$ regime has been a topic of several works in recent years.
Convergence of the costs $\OTep$ to $\OT$ with precise rates was investigated in \cite{pal2019difference,chizat2020faster,conforti2021formula}. The works \cite{carlier2017convergence, leonard2012schrodinger, bernton2021entropic, ghosal2021stability} study the convergence of the minimizers $\pi_\eps$ to $\pi_0$ under varying assumptions. Convergence of the potentials in a very general setting was established in \cite{nutz2021entropic}, though without a rate of convergence in $\eps$. In the semi-discrete case, this gap was closed in \cite{altschuler2021asymptotics} followed closely by \cite{delalande2021nearly}, which gave non-asymptotic rates. The Sinkhorn Divergence, a non-negative, symmetric version of $\OTep$, was introduced in \cite{genevay2018learning}, was statistically analysed in ~\cite{goldfeld2022limit} and also in \cite{gonzalez2022weak,del2022improved}, and was connected to the entropic Brenier map in \cite{pooladian2022debiaser}. The recent pre-print by \cite{rigollet2022sample} proved parametric rates of estimation between the empirical entropic Brenier map and its population counterpart, though with an exponentially poor dependence on the regularization parameter (see \cref{remark: rigollet2022}). 
Using covariance inequalities, the entropic Brenier potentials were used give a new proof of Caffarelli's contraction theorem; see \cite{chewi2022entropic}; this approach was recently generalized in \cite{conforti2022weak}. Entropic optimal transport has also come into contact with the area of deep generative modelling through the following works \cite{finlay2020learning,de2021diffusion}, among others.

\section{Statistical performance of the entropic estimator in the semi-discrete setting}\label{sec:main_results}
Let $P_n$ and $Q_n$ be the empirical measures associated with two $n$-samples from $P$ and $Q$. 
We make the following regularity assumptions on $P$, already introduced by \citep{delalande2021nearly}.

\begin{itemize}
\item[\textbf{(A)}] The measure $P$ has a compact convex support $\Omega \subseteq B(0;R)$, with a density $p$ satisfying $0<p_{\min}\leq p \leq p_{\max}<\infty$ for positive constants $p_{\min}$, $p_{\max}$ and $R$.
\end{itemize}
For example, $P$ can be the uniform distribution over $\Omega$, or a truncated Gaussian distribution. Furthermore, we will need the following assumption on $Q$.
\begin{itemize}
    \item[\textbf{(B)}] The discrete probability measure $Q=\sum_{j=1}^J q_j \delta_{y_j}$ is such that $q_j\geq q_{\min}>0$ and $y_j\in B(0;R)$ for all $j\in [J]$.
\end{itemize}

The goal of this section is to prove the following theorem:
\begin{theorem}\label{thm: mapest_semidiscrete}
Let $P$ satisfy \textbf{(A)} and let $Q$ satisfy \textbf{(B)}. Let $\hat T_\eps = T_\eps^{P_n \to Q_n}$.
Then, for $\eps \asymp n^{-1/2}$ and $n$ large enough,
    \begin{align}
        \E\| \hat{T}_\eps - T_0\|^2_{L^2(P)} \lesssim n^{-1/2}\,.
    \end{align}
\end{theorem}
\begin{remark}We remark that the hidden constants in \cref{thm: entmap_semidiscrete_samp_full} and related results depend on $J, p_{\min}, p_{\max}, q_{\min}$ and $R$.
\end{remark}

\begin{remark}[Fixing the support via rounding]
At present, the entropic map need not necessarily map exactly to one of $\{y_1,\ldots,y_J\}$. In fact, $ \hat{T}_\eps : \R^d \to \text{conv}(\{Y_1,\ldots,Y_n\})$, where $\text{conv}(A)$ is the convex hull for some set $A$. In turn, the support of the entropic map does not in general match that of $Q$. However, this can be readily fixed with a rounding scheme.
We can replace our estimator by $\bar T_\eps$ which is obtained by mapping the output of $\hat T_\eps$ to its nearest neighbor in the support of $Q$ -- this projection step is easy to compute, given that we essentially know the support of $Q$ via samples. 
By viewing this as a projection onto an appropriate set (namely, the set of transport maps with codomain equal to the support of $Q$), and applying the triangle inequality, it holds that
\begin{align*}
 \E \|\bar{T}_\eps - T_0\|^2_{L^2(P)} \leq 2 \E \|\hat{T}_\eps - T_0\|^2_{L^2(P)}    
\end{align*}
but $\bar{T}_\eps$ matches the support of $Q$. 
 \end{remark}

Let $T_\eps= T_\eps^{P\to Q}$ denote the entropic Brenier map associated to $P$ and $Q$. 
Our proof relies on the following bias-variance decomposition:
\begin{align*}
    \E\|\hat{T}_{\eps} - T_0\|^2_{L^2(P)}\hspace{-.05cm} &\lesssim \E\|\hat{T}_{\eps} - T_\eps\|^2_{L^2(P)} \hspace{-.05cm} +\hspace{-.05cm} \|T_\eps - T_0\|^2_{L^2(P)}\hspace{-.05cm} \,.
\end{align*}
Following the next two results (\cref{thm: mapest_semidiscrete_population} and \cref{thm: entmap_semidiscrete_samp_full}) and the preceding decomposition, the proof of \Cref{thm: mapest_semidiscrete} is merely a balancing act in the regularization parameter $\eps$.
\begin{theorem}\label{thm: mapest_semidiscrete_population}
    Let $P$ satisfy \textbf{(A)} and let $Q$ satisfy \textbf{(B)}. Then, for $\eps$ small enough,
    \begin{align}
        \|T_\eps - T_0\|^2_{L^2(P)} \lesssim \eps\,. 
    \end{align}
\end{theorem}

The proof of \Cref{thm: mapest_semidiscrete_population} relies on the following qualitative picture: if a point $x$ belongs to some Laguerre cell $L_j$, and is far away from the boundary of $L_j$, then the entropic optimal plan $\pi_\eps$ will send almost all of its mass towards the point $y_j=T_0(x)$, sending an exponentially small amount of mass to the other points $y_j$. Such a picture is correct as long as $x$ is at distance at least $\eps$ from the boundary of the Laguerre cell $L_j$, incurring a total error  of order $\eps$. A rigorous proof of \cref{thm: mapest_semidiscrete_population} can be found in \Cref{sec:approx}.

Note that this rate is slower than the rate appearing in \Cref{thm: pnw_main_thm} in the continuous-to-continuous case. The following example shows that the dependency in $\eps$ is optimal in \Cref{thm: mapest_semidiscrete_population}, indicating that the presence of discontinuities necessarily affects the approximation properties of the entropic Brenier map.
\begin{example}\label{example: pop_lower_bound}
Let $P$ be a probability measure on $\R$ having a  symmetric bounded density $p$ continuous at $0$, and let $Q=\frac{1}{2}(\delta_{-1}+\delta_1)$. Following \citep[Section 3]{altschuler2021asymptotics}, one can check that the entropic Brenier map in this setting is the following scaled sigmoidal function
\begin{equation*}
    T_\eps(x) = \tanh(2x/\eps)\,,
\end{equation*}
whereas the optimal transport map $T_0(x) = \text{sign}(x)$.
Then, performing a computation
\begin{align*}
    \|T_\eps - T_0\|^2_{L^2(P)} &= 2 \int_0^\infty (1-\tanh(2x/\eps))^2 p(x) \dd x \\
    &= \eps \int_0^\infty(1-\tanh(u))^2 p(u\eps/2) \dd u \\
    &= \eps p(0) (\log(4)-1) + o(\eps)\,,
\end{align*}
where in the last step we invoked the dominated convergence theorem, and computed the limiting integral.
\end{example}

\begin{remark}
    Assumption \textbf{(A)} can be relaxed for \Cref{thm: mapest_semidiscrete_population} to hold. More precisely, it can be replaced by Assumptions 2.2 and 2.9 of \citep{altschuler2021asymptotics}, that hold for unbounded measures such as the normal distribution.
\end{remark}

Finally, we present the sample-complexity result: 
\begin{theorem}\label{thm: entmap_semidiscrete_samp_full}
    Let $P$ satisfy \textbf{(A)}  and let $Q$ satisfy \textbf{(B)}. Then, for $0<\eps\leq 1$ such that $\log(1/\eps) \lesssim n/\log(n)$
    \begin{align}
        \E\| \hat{T}_{\eps} - T_\eps \|^2_{L^2(P)} \lesssim \eps^{-1} n^{-1}\,. 
    \end{align}
\end{theorem}
\begin{remark}\label{remark: rigollet2022}
In \cite{rigollet2022sample}, the authors show that if $P$ and $Q$ are merely compactly supported with $\supp(P),\supp(Q) \sse B(0;R)$, then
\begin{align}
    \E \|\hat{T}_{\eps} - T_\eps \|^2_{L^2(P)} \lesssim e^{cR^2/\eps}  \eps^{-1}n^{-1}\,,
\end{align}
where $c > 0$ is some absolute positive constant. Thus, under the additional structural assumptions of the semi-discrete formulation, we are able to significantly improve the rate of convergence between the empirical and population entropic Brenier maps. 
\end{remark}

The proof of \cref{thm: entmap_semidiscrete_samp_full} relies on a novel stability result, reminiscent of \citep[Theorem 6]{manole2021plugin}, which is of independent interest. We provide the proof in \Cref{app: stability}.

\begin{proposition}\label{prop: entmap_stability}
	
	Let $\mu, \nu, \mu', \nu'$ be four probability measures supported in $B(0; R)$.
	Then the entropic maps $T_\ep^{\mu \to \nu}$ and $T_{\ep}^{\mu' \to \nu'}$ satisfy
	\begin{align*}
	\frac{\eps}{8 R^2}\|T_\ep^{\mu \to \nu} - T_\ep^{\mu' \to \nu'}\|_{L^2(\mu)}^2 \leq \int (\varphi_\eps^{\mu'\to \nu'} - \varphi_\eps^{\mu\to \nu}) \dd \mu + \int (\psi_\eps^{\mu'\to \nu'} - \psi_\eps^{\mu\to \nu}) \dd \nu + \eps \kl{\nu}{\nu'}
	\end{align*}
\end{proposition}
\begin{remark}
The right side of the bound in \cref{prop: entmap_stability} is equal to
\begin{align*}
	S_\eps(\mu, \nu) - S_\eps(\mu', \nu') + \int f_\ep^{\mu' \to \nu'} \dd (\mu' - \mu) + \int g_\ep^{\mu' \to \nu'} \dd (\nu' - \nu) + \eps \kl{\nu}{\nu'}\,,
\end{align*}
where $f_\ep^{\mu' \to \nu'} = \frac 12 \| \cdot \|^2 - \phi_\ep^{\mu' \to \nu'}$ and $g_\ep^{\mu' \to \nu'} = \frac 12 \| \cdot \|^2 - \psi_\ep^{\mu' \to \nu'}$.
\Cref{prop: entmap_stability} is therefore the entropic analogue of the stability bounds of \citep[Theorem 6]{manole2021plugin} and \citep[Lemma 5.1]{GhoSen22}.
Unlike those results, \cref{prop: entmap_stability} allows both the source and target measure to be modified, and does not require any smoothness assumptions.
\end{remark}
\subsection*{Proof sketch of \cref{thm: entmap_semidiscrete_samp_full}}
To prove \Cref{thm: entmap_semidiscrete_samp_full}, we first consider the \textit{one-sample setting}, where we assume that we only have access to samples $Y_1,\dots,Y_n\sim Q$, but we have full access to $P$. We then consider the one-sample entropic estimator $T_\ep^{P\to Q_n}$.  
We apply \Cref{prop: entmap_stability} with $\mu=\mu'\defeq P$,  $\nu \defeq Q_n$ and $\nu' \defeq Q$, yielding (see \Cref{cor:same_mu} for details)
\begin{align*}\label{eq: apply_stab_one_sample}
\begin{split}
   \frac{\eps}{8 R^2} \E\|T^{P \to Q_n}_\ep - T_\eps\|_{L^2(\mu)}^2 \leq \E\Big( \int (\psi_\eps - \psi_\eps^{P\to Q_n}) \dd(Q_n- Q) +  \eps\kl{Q_n}{Q}\Big)\,.
   \end{split}
\end{align*}

Let $\chis{P}{Q}$ denote the $\chi^2$-divergence between probability measure. Young's  inequality (see \Cref{lem:young}) and the inequality $\kl{Q_n}{Q}\leq \chis{Q_n}{Q}$  yield  the following bound:
\begin{align*}
    \E\|T_{\eps}^{P\to Q_n} - T_\eps\|^2_{L^2(P)} \leq  \frac{8R^2}{\eps}\Big( \frac{\E [\mathrm{Var}_Q(\psi_{\eps}^{P\to Q_n} - \psi_\eps)]}{2} + \frac{\E [\chis{Q_n}{Q}]}{2} \Big)  + 8R^2 \E [\chis{Q_n}{Q}]\,.
\end{align*}
 
To complete our proof sketch, we use a new stability result on the entropic dual Brenier potentials, catered for the semi-discrete setting.
\begin{proposition}\label{prop:stability_potentials}
Let $\mu$ be a measure that  satisfies \textbf{(A)}. 
    Let $\nu$, $\nu'$ be two discrete probability measures supported on $\{y_1,\dots,y_J\}$, with $\nu' \geq \lambda \nu$ for some $\lambda>0$. 
    Then, for $0< \eps \leq 1$,
    \begin{equation}
        \mathrm{Var}_{\nu}(\psi_\eps^{\mu\to \nu'}-\psi_\eps^{\mu\to \nu}) \leq \frac{C}{\lambda^2} \chis{\nu'}{\nu},
    \end{equation}
    where $C$ depends on $R$, $p_{\min}$ and $p_{\max}$.
\end{proposition}
Moreover, a computation provided in  \Cref{lem: kl_expectation_bound} shows that $\E[\chis{Q_n}{Q}] = \frac{J-1}{n}$, which is enough to conclude the proof of the one-sample case, see \Cref{sec: proofs in semidiscrete case} for details. 

The two-sample setting is tackled using similar reasoning, where we ultimately prove in \Cref{sec: proofs_two_sample}  that the risk $\E\|\hat{T}_{\eps} - T_{\eps}^{P\to Q_n}\|^2_{L^2(P)}$ is upper bounded by
\begin{align*}
    \frac{8R^2}{\eps}\E\int (\phi_{\eps}^{P\to Q_n} - \phi_{\eps}^{P_n\to Q_n})\dd (P_n-P)\,.
\end{align*}
Such a quantity can again be related to the estimation of the dual potentials $\psi_{\eps}^{P\to Q_n}$ and $\psi_{\eps}^{P_n\to Q_n}$. Using the same reasoning as before, we expect a parametric rate of convergence for this term as well. Merging the two results completes the proof of \cref{thm: entmap_semidiscrete_samp_full}. We refer to \Cref{sec: proofs_two_sample}  for full details.

\section{Comparing against the 1NN estimator}\label{sec: num_all}
\subsection{Rate optimality of the entropic Brenier map}
The upper bound of \cref{thm: entmap_semidiscrete_samp_full} shows that our estimator achieves the $n^{-1/2}$ rate.
In fact, the following simple proposition tells us that this rate is optimal in the semi-discrete case.

\begin{proposition}\label{prop:minimax}
Let $P$ be the uniform distribution on $[-1/2,1/2]^d$ and for any $J \geq 2$, let $\cQ_{J}$ denote the space of of probability measures with at most $J$ atoms, supported on $[-1/2,1/2]^d$. Define the minimax rate of estimation 
\begin{equation*}
    \cR_n(\cQ_{J}) = \inf_{\hat T} \sup_{Q\in\cQ_{J}} {\E_{Q^{n} }}[\|\hat T-T_0^{P\to Q}\|^2_{L_2(P)}]\,.
\end{equation*}
Then, it holds that $\cR_n(\cQ_{J}) \geq n^{-1/2}/64$. 
\end{proposition}
\vspace{-.4cm}
\begin{proof}
Let $e$ be a vector of the canonical basis of $\R^d$, scaled by $1/2$.
Fix $0<r<1/2$ and let $Q_0 = \frac{1}{2}\delta_{-e} + \frac{1}{2}\delta_e$ and $Q_1= (\frac 12 -r)\delta_{-e} + (\frac 12 +r)\delta_e$.  A computation gives $\|T_0^{P\to Q_0}- T_0^{P\to Q_1}\|^2_{L_2(P)} = r$. Therefore, by Le Cam's lemma~\citep[see, e.g.,][Chapter 15]{wainwright2019high},
\begin{equation}
    \cR_n(\cQ_{J,R})\geq  \frac{r}{8}(1-{\tv{Q_0^{n}}{Q_1^{n}})}.
\end{equation}
Let $\hel{Q_0}{Q_1}$ denote the (squared) Hellinger distance between measures. We have  $$\tv{Q_0^n}{Q_1^n}^2\leq \hel{Q_0^n}{Q_1^n}\leq  n \hel{Q_0}{Q_1}\,.$$ Furthermore, a computation gives 
\begin{align*}
    \hel{Q_0}{Q_1} &= \bigg(\sqrt{\tfrac12 - r} - \sqrt{\tfrac12}\bigg)^2 + \bigg(\sqrt{\tfrac12 + r} - \sqrt{\tfrac12}\bigg)^2 \\
    &=2 - (\sqrt{1+2r}+\sqrt{1-2r})\\
    & \leq 4 r^2.
\end{align*}
We obtain the conclusion by picking $r=n^{-1/2}/4$.
\end{proof}

\subsection{The 1NN estimator is proveably suboptimal}
The 1-Nearest-Neighbor estimator, henceforth denoted $\Tnn$, was proposed by \cite{manole2021plugin} as a computational surrogate for estimating optimal transport maps in the low smoothness regime. Written succinctly, their estimator is $\Tnn(x) = \sum_{i=1}^n \bm{1}_{V_{i}}(x) Y_{\hat{\pi}(i)}$,
where $(V_i)_{i=1}^n$ are Voronoi regions i.e.
\begin{align*}
    V_i \defeq \{ x \in \Rd \ : \ \|x - X_i\| \leq \| x - X_k\|\,, \forall \ k \neq i\}\,,
\end{align*}
and $\hat{\pi}$ is the optimal transport plan between the empirical measures $P_n$ and $Q_n$, which amounts to a permutation. Computing the closest $X_i$ to a new sample $x$ has runtime $\cO(n\log(n))$, though the complexity of this estimator is determined by computing the plan $\hat{\pi}$, which takes $\cO(n^3)$ time via, e.g., the Hungarian Algorithm \citep[see][Chapter 3]{PeyCut19}. 

When $\phi_0$ is smooth and strongly convex, \citep{manole2021plugin} showed that, for $d \geq 5$,
\begin{align*}
    \E\|\Tnn - \nabla \phi_0\|^2_{L^2(P)}\lesssim n^{-2/d}\,.
\end{align*}
\vspace{-.7cm}

In contrast to the rate optimality of the entropic Brenier map, we now show that $\Tnn$ is proveably suboptimal in the semi-discrete setting. 
Not only does it fail to recover the minimax rate obtained by the entropic Brenier map, but its performance in fact degrades in comparison to the smooth case. A proof appears in \cref{1nn_lb}.
\begin{proposition}\label{prop:1NN_suboptimal}
    There exist a measure $P$ satisfying \textbf{(A)} and a discrete measure ${Q}$ satisfying \textbf{(B)} such that for $d\geq 3$
    \begin{align*}
    \E\|\Tnn - T_0^{P \to Q}\|^2_{L^2(P)}\gtrsim n^{-1/d}\,.
\end{align*}
\end{proposition}

\subsection{Experiments}\label{sec: experiments}
We briefly verify our theoretical findings on synthetic experiments. To create the following plots, we draw two sets of $n$ i.i.d. points from $P$, $(X_1,\ldots,X_n)$ and $(X_1',\ldots,X_n')$, and create target points $Y_i = T_0(X_i')$, where $T_0$ is known to us in advance in order to generate the data. Our estimators are computed on the data $(X_1,\ldots,X_n)$ and $(Y_1,\ldots,Y_n)$, and we evaluate the Mean-Squared error criterion 
$$\text{MSE}(\hat{T}) = \| \hat{T}- T_0\|_{L^2(P)}^2$$
 of a given map estimator $\hat{T}$ using Monte Carlo integration, using 50000 newly sampled points from $P$. We plot the means across 10 repeated trials, accompanied by their standard deviations. 
 
\subsubsection{Semi-discrete example \#1}
First consider $P = \text{Unif}([0,1]^d)$ and create atoms $\{y_1,\ldots,y_J\}$ by partitioning the points along the first coordinate for all $j \in [J]$:
\begin{align*}
    (y_j)[1] = \frac{(j - 1/2)}{J} \,, \quad (y_j)[2] = \cdots = (y_j)[d] = 0.5\,.
\end{align*}
We choose uniform $q_j = 1/J$ for $j \in [J]$. In this case, it is easy to see that the optimal transport map $T_0(x)$ is uniquely defined by the first coordinate of $x_1$. \Cref{fig: semidiscrete_J2} illustrates the rate-optimal performance of the entropic Brenier map, and the proveably suboptimal performance of the 1-Nearest-Neighbor estimator.

\begin{figure*}[t!]
	\center
\includegraphics[width=0.4\textwidth]{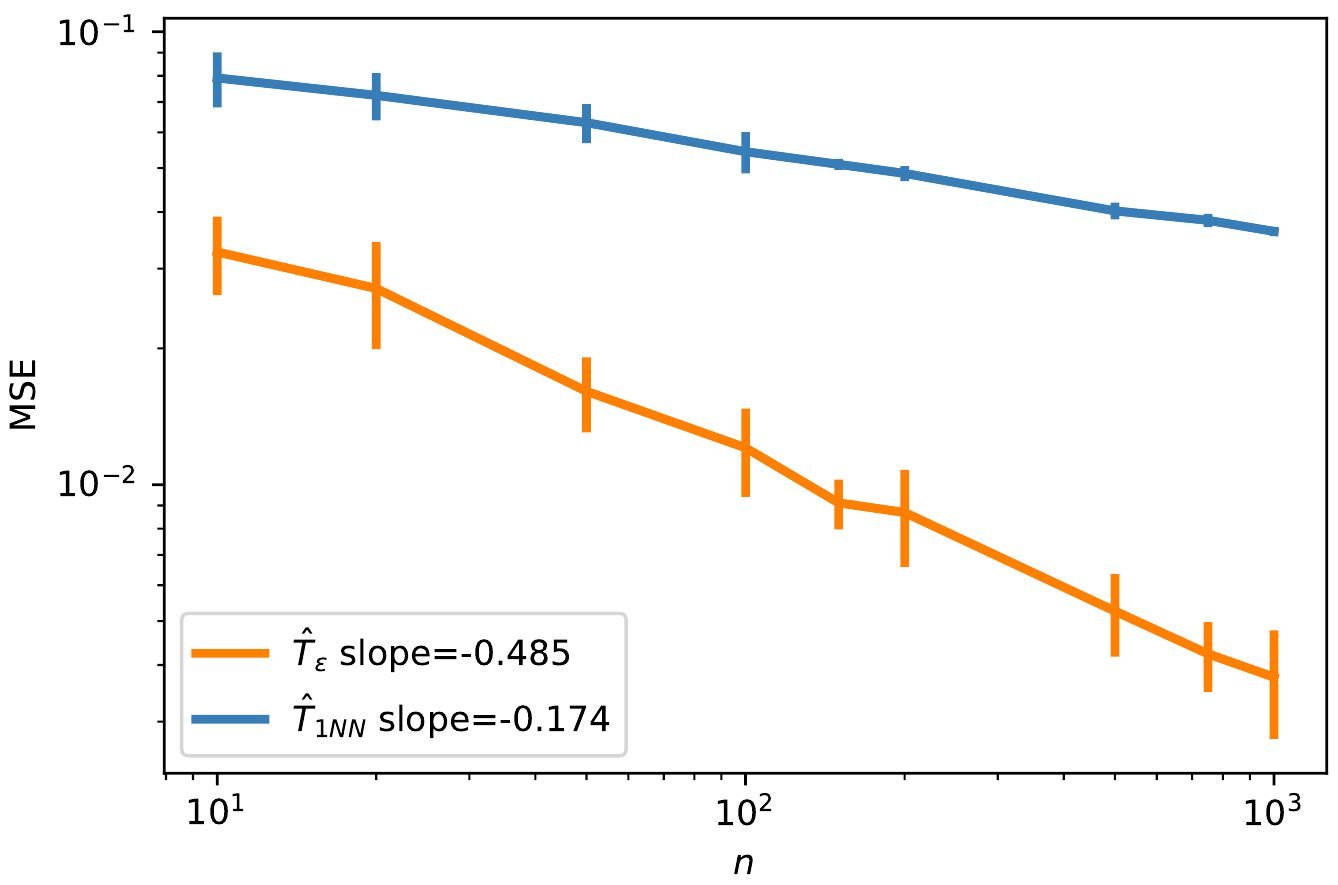} \qquad\qquad
\includegraphics[width=0.4\textwidth]{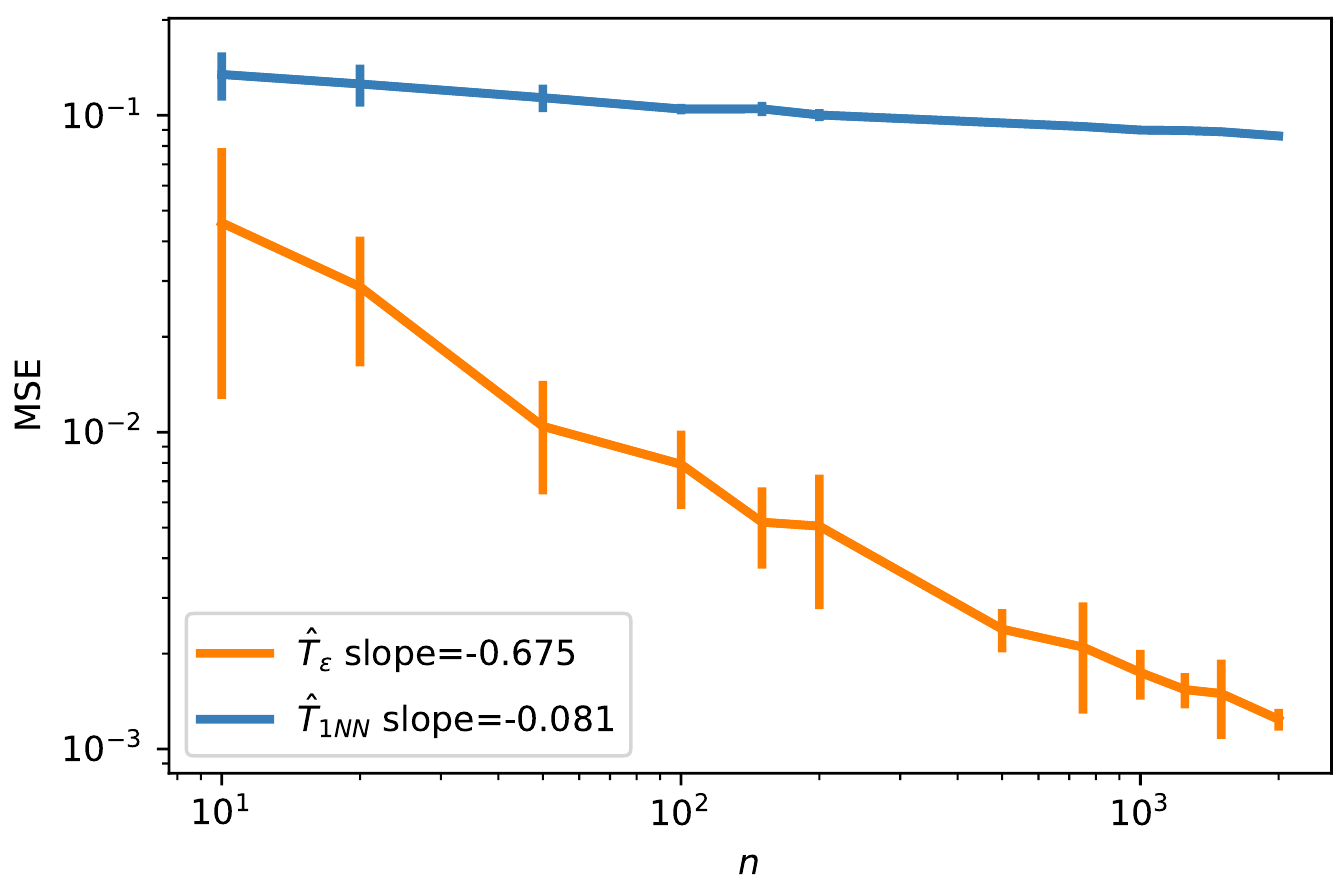}
\vskip -0.15in
\caption{Left: $\hat{T}_{\eps}$ versus $\hat{T}_{\text{1NN}}$ for $J = 2$ and $d=10$. Right: $\hat{T}_{\eps}$ versus $\hat{T}_{\text{1NN}}$ for $J = 10$ and $d=50$.}
\label{fig: semidiscrete_J2}
 \end{figure*}

\subsubsection{{Semi-discrete example \#2}}
We now consider a synthetic experiment with far less symmetry. Let $P = \text{Unif}([0,1]^d)$, and fix $J \in \NN$. We randomly generate $y_1,\ldots,Y_J \in [0,1]^d$, and also randomly generate $\psi_0 \in \R^J$, and consider the optimal transport map $T_0(x) = \argmin_{j \in [J]}\{x^\top y_j - (\psi_0)_j\}$. We define $Q = (T_0)_\sharp P$, leading to the same setup as before, but with a less structured optimal transport map. We consider $J = 5$ and $d=50$, and repeat the procedure of the preceding section to generate our data, and the resulting estimator. \cref{fig: random_d50} contains plots the MSE as a function of $n$, where again we see a log-linear slope of around $-0.5$, which agrees with our theory.

\begin{figure}[ht]
\begin{center}
\centerline{\includegraphics[width=0.55\columnwidth]{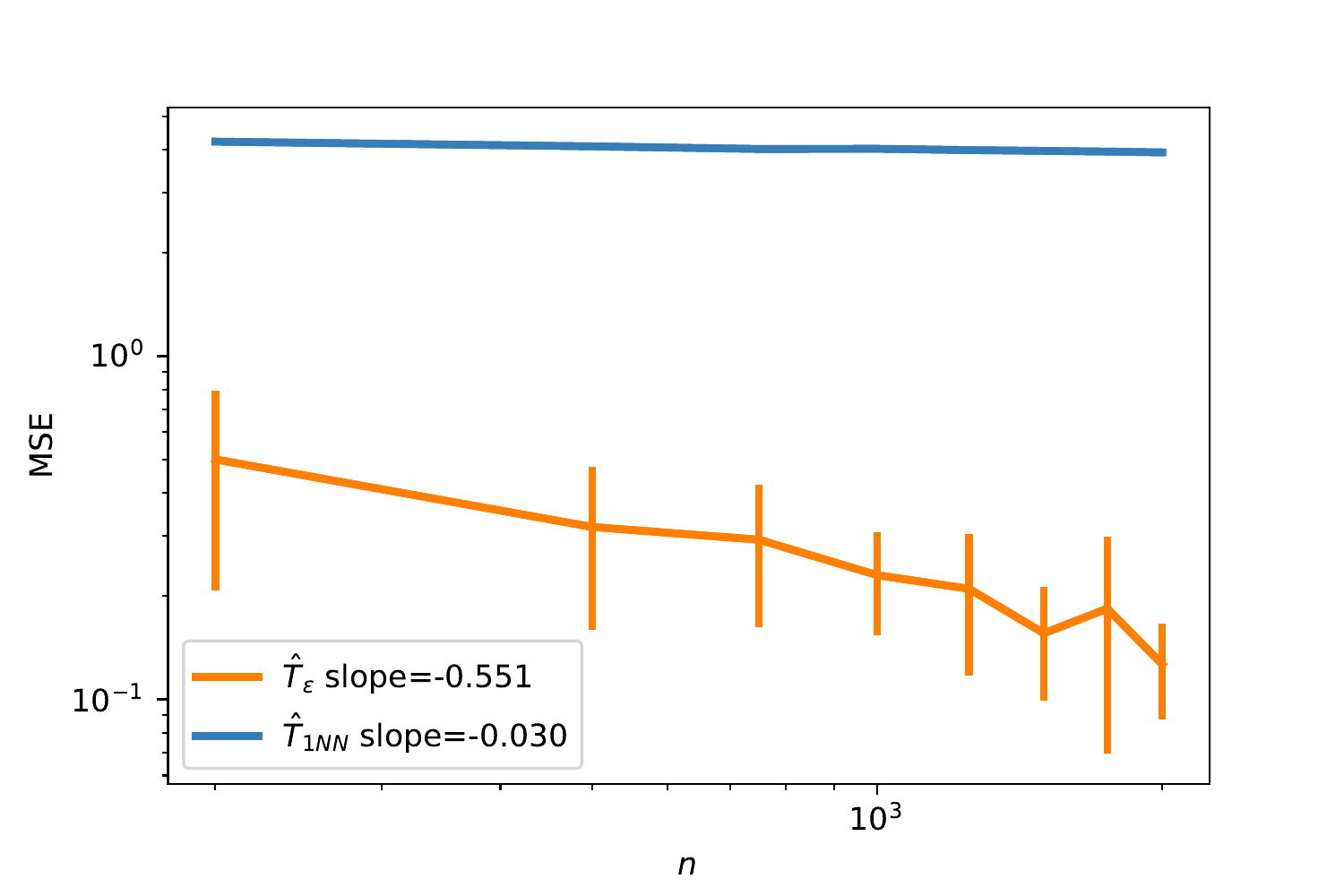}}
\vskip -0.2in
\caption{$\hat{T}_{\eps}$ versus $\hat{T}_{\text{1NN}}$ for with $\psi_0$ random in $d=50$}
\label{fig: random_d50}
\end{center}
\end{figure}

\subsubsection{Discontinuous example}
We turn our attention to a discontinuous transport map, where for $x \in \R^d$, all the coordinates are fixed except for the first one
\begin{align*}
    T_0(x) = 2\text{sign}(x[1]) \otimes x[2] \otimes \cdots  \otimes x[d]\,.
\end{align*} 
We choose $P = \text{Unif}([-1,1]^d)$ to exhibit a discontinuity in the data. Focusing on $d=10$, we see in \Cref{fig: splitting_d10} that the entropic map estimator avoids the curse of dimensionality and enjoys a faster convergence rate, with better constants.
\begin{figure}[ht]
\begin{center}
\centerline{\includegraphics[width=0.55\columnwidth]{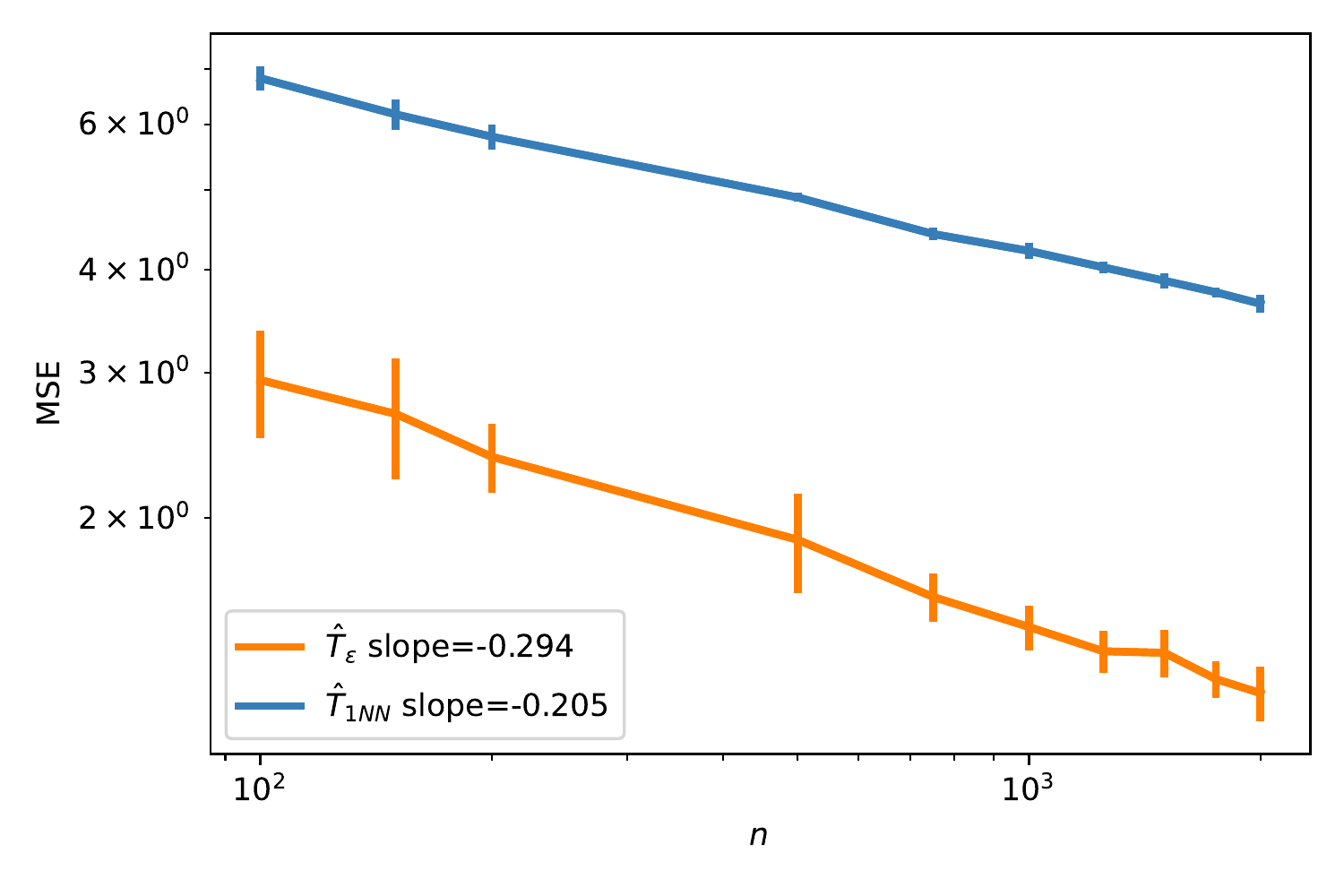}}
\vskip -0.2in
\caption{$\hat{T}_{\eps}$ versus $\hat{T}_{\text{1NN}}$ for $d=10$}
\label{fig: splitting_d10}
\end{center}
\end{figure}

\section{Conclusion}
Understanding optimal transport maps in the semi-discrete case is a natural stepping-stone to understanding the case for general discontinuous transport maps. In this work, we propose a tractable, minimax optimal estimator of the Brenier map in the semi-discrete setting, where the rate of estimation is dimension independent.
To prove our result, we require several new results and techniques, and, as a by-product of our analysis, give the first parametric rates of estimation the entropic Brenier map, without exponential dependence in the regularization parameter. Our synthetic experiments indicate that the entropic Brenier map might be useful in estimating other variants of discontinuous transport maps, which constitutes  an interesting direction for future research.

\section*{Acknowledgements}
AAP would like to thank Tudor Manole for fruitful discussions, and gratefully thanks funding sources NSF Award 1922658, and Meta AI Research. JNW is supported by the Sloan Research Fellowship and NSF grant DMS-2210583. We thank the anonymous reviewer for suggesting the addition of a rounding scheme.

\newpage
\appendix
\onecolumn

\section{Reminders on semi-discrete entropic optimal transport}\label{sec:reminder}
We recall in this  section some known results on entropic optimal transport that will be needed later. Let $\mu,\nu\in \cP(\Omega)$, where $\Omega\subset B(0;R)$ is a compact set.

\begin{lemma}[\citealp{genevay2019sample}]\label{lem:bounded_amplitude}
    The entropic potential $(\phi_\eps^{\mu\to\nu},\psi_\eps^{\mu\to\nu})$ have a bounded amplitude, in the sense that
    \begin{equation}
        \max_{x\in \Omega} \phi_\eps^{\mu\to\nu} - \min_{x\in \Omega} \phi_\eps^{\mu\to \nu} \leq cR
    \end{equation}
    for some absolute constant $c$, and similarly for $\psi_\eps^{\mu\to\nu}$.
\end{lemma}

Assume now that $\nu= \sum_{j=1}^J \nu_j\delta_{y_j}$ is a discrete measure. In this situation, only the values of the dual potential $\psi_\eps^{\mu\to\nu}$ on the points $y_1,\dots,y_J$ are relevant. We therefore consider $\psi_\eps^{\mu\to\nu}$ as a vector in $\R^J$. The potentials $\phi_\eps^{\mu\to\nu}$ and $\psi_\eps^{\mu\to\nu}$ are dual of one another, in the sense of the $\eps$-Legendre transform. Given a finite measure $\rho$, the $\eps$-Legendre transform of a function $h$ with respect to $\rho$ is given by
\begin{equation}
   \Phi_\eps^\rho( h)(x) = \eps \log \int e^{(\langle x,y\rangle -h(x))/\eps} \dd \rho(x).
\end{equation}
Relations \eqref{eq: opt_cond_1} and \eqref{eq: opt_cond_2} express that $\phi_\eps^{\mu\to\nu} = \Phi_\eps^\nu(\psi_\eps^{\mu\to\nu})$ and vice-versa.  In the semi-discrete setting, it is also convenient to introduce the $\eps$-Legendre transform with respect to the counting measure $\sigma$ on $\{y_1,\dots,y_J\}$. For a vector $\psi\in\R^J$, we have
\begin{equation}\label{eq:shifted_potential}
    \Phi_\eps(\psi)(x) \defeq \Phi_\eps^\sigma(\psi)(x) = \eps\log \sum e^{(\langle x,y_j\rangle -\psi(y_j))/\eps}.
\end{equation}
The $\Phi_\eps$ transform and the $\Phi_\eps^\nu$ transform are linked through the relation
\begin{equation}
  \Phi_\eps^\nu(\psi) =\Phi_\eps(\tilde\psi) \qquad \text{ where }\qquad \tilde \psi(y_j) = \psi(y_j)-\eps \log \nu_j\,,
\end{equation}
where we call $\tilde\psi$ a \emph{shifted} potential. With this notation, the optimality condition on the potentials can be rephrased. Let 
\begin{equation}
    F_\eps^{\mu\to\nu}:\psi\in \R^J\to \int \Phi_\eps(\psi) \dd \mu + \int \psi \dd \nu\,.
\end{equation}
Then, the function $F_\eps^{\mu\to\nu}$ is minimized at $\tilde \psi_\eps^{\mu\to\nu}$. 
For $\psi\in \R^J$ and $x\in \R^d$, we introduce the probability measure supported on $\{y_1,\dots,y_J\}$ given by
\begin{equation}\label{eq:def_pi_eps_x}
 \forall i\in [J],\quad   \pi_\eps^x[\psi](y_i) = \frac{e^{(\langle x,y_i\rangle -\psi(y_i))/\eps}}{\sum_{j=1}^Je^{(\langle x,y_j\rangle -\psi(y_j))/\eps}} = e^{(\langle x,y_i\rangle -\Phi_\eps(\psi)(x)-\psi(y_i))/\eps}.
\end{equation}
A computation gives $\nabla F_\eps^{\mu\to\nu}(\psi) = \int    \pi^x_\eps[\psi]\dd \mu(x) - \nu$, so that at optimality, we have
\begin{equation}\label{eq:desintegration}
    \int    \pi^x_\eps[\tilde\psi_\eps^{\mu\to\nu}]\dd \mu(x) = \nu.
\end{equation}
In this case, $\pi^x_\eps=\pi^x_\eps[\tilde \psi_\eps^{\mu\to\nu}]$ is the conditional  distribution of the second marginal of $\pi_\eps$ given that the first is equal to $x$, as in \Cref{sec:entropic_brenier_map}. More generally, for any potential $\psi$, the first order condition implies that $\psi$ is equal to $\tilde \psi_\eps^{\mu\to\nu_\psi}$, the optimal dual potential between $\mu$ an $\nu_\psi = \int    \pi^x_\eps[\psi]\dd \mu(x)$. 

\section{Bound on the approximation error}\label{sec:approx}

\begin{proof}[Proof of \Cref{thm: mapest_semidiscrete_population}]
Let $i,j\in[J]$. We define the $j$th slack at $x\in L_i$ by
\begin{equation}
    \frac 12 \Delta_{ij}(x)  = -\langle x,y_j\rangle +\phi_0(x) +\psi_0(y_j).
\end{equation}
As $\phi_0$ is the Legendre transform of $\psi_0$, we have $\Delta_{ij}(x)\geq 0$. If the cells $L_i$ and $L_j$ have a nonempty intersection, the set $H_{ij}(t) = \{x\in L_i:\ \Delta_{ij}(x)=t\}$ represents the trace on $L_i$ of the hyperplane spanned by the boundary between $L_i$ and $L_j$, shifted by $t$. 
It is stated in \cite{altschuler2021asymptotics} that for every nonnegative measurable function $f:\R\to {\R_+}$,
\begin{equation}
    \int_{L_i} f(\Delta_{ij}(x)) p(x) \dd x = \frac{1}{2\|y_i-y_j\|} \int_0^\infty f(t) h_{ij}(t) \dd t,
\end{equation}
where $h_{ij}(t) =\int_{H_{ij}(t)} p(x) \dd \cH_{d-1}(x)$ and $ \cH_{d-1}$ is the $(d-1)$-dimensional Hausdorff measure. In particular, $w_{ij}=h_{ij}(0)$ is the (weighted) surface of the boundary between the $i^{\text{th}}$ and $j^{\text{th}}$ Laguerre cells (should it exist). Given $x\in L_i$, let $s(x) = \min_{j\neq i} \frac 12 \Delta_{ij}(x)$. When the point $x$ is sufficiently inside its Laguerre cell, the conditional probability $\pi_\eps^x$ becomes extremely concentrated around the point $y_i$, as the next lemma shows. Note that $\pi_0^x = \delta_{y_i}$ when $x\in L_i$.

\begin{lemma}
    Let $x \in L_i$. For $\eps$ small enough, it holds that for every $j\in [J]$, $|\pi_\eps^x(y_j)-\pi_0^x(y_j)|\leq ce^{-s(x)/\eps}$, where $c$ depends on $J$, the distances $\|y_i-y_j\|$ and on the quantities $w_{ij}$.
\end{lemma}
Such a result was already stated in \citep[Corollary 2.2]{delalande2021nearly}, although while requiring that the source measure $P$ has a H\"older continuous density. Only assumption \textbf{(A)} is needed here.

\begin{proof}
According to
\citep[Proposition 4.6]{altschuler2021asymptotics}, for $\eps$ small enough, 
\begin{equation}\label{eq:altchuler_lemma}
    \eps^{-1}\|\tilde \psi_\eps-\psi_0\|_\infty\leq C,
\end{equation}
where $\tilde\psi_\eps$ is the shifted version of $\psi_\eps$ (see \eqref{eq:shifted_potential}) and $C$ depends on  the distances $\|y_i-y_j\|$ and on the $w_{ij}$s. 
Following \citep[Proof of Corollary 2.2]{delalande2021nearly} and \eqref{eq:def_pi_eps_x},  we have for $j\neq i$ 
\begin{align*}
    |\pi_\eps^x(y_j)-\pi_0^x(y_j)| &= \pi_\eps^x(y_j) =  \frac{e^{(\langle  x,y_j \rangle -\tilde \psi_{\eps}(y_j))/\eps}}{\sum_{j'=1}^Je^{(\langle x,y_{j'}\rangle -\tilde \psi_{\eps}(y_{j'}))/\eps}} \leq e^{2C}\frac{e^{(\langle  x,y_j \rangle - \psi_{0}(y_j))/\eps}}{\sum_{j'=1}^Je^{(\langle x,y_{j'}\rangle -\psi_{0}(y_{j'}))/\eps}} \leq e^{2C} e^{-s(x)/\eps}.
\end{align*}
A similar computation yields that $|\pi_\eps^x(y_i)-\pi_0^x(y_i)|=|\pi_\eps^x(y_i)-1|\leq Je^{2C}e^{-s(x)/\eps}$. \qedhere
    
\end{proof}
We can bound for any $x\in L_i$,
\begin{equation}
    \|T_\eps(x) - T_0(x)\| = \| \sum_{j=1}^J y_j (\pi_\eps^x(y_j)-\pi_0^x(y_j))\|\leq  c\sum_{j=1}^J\|y_j\| e^{-s(x)/\eps}.
\end{equation}
Therefore, letting $C'$ denote a constant, which may depend on $J$, whose value may change from line to line, we obtain
\begin{align}
    \|T_\eps-T_0\|^2_{L_2(P)} &= \sum_{i=1}^J \int_{L_i} \|T_\eps(x) - T_0(x)\|^2\dd P(x) \leq C' \sum_{i=1}^J \int_{L_i} \sum_{j = 1}^J e^{-2s(x)/\eps} \dd P(x) \\
    &\leq C' \sum_{i\neq j} \int_{L_i} e^{-\Delta_{ij}(x)/\eps} \dd P(x) \leq C' \sum_{i \neq j} \frac{1}{2\|y_i-y_j\|} \int_0^\infty e^{-t/\eps} h_{ij}(t)\dd t\,,\label{eq:last_approx}
\end{align}
where in the second equality, we used the definition of $s(x)$. Assumption \textbf{(A)} ensures that the functions  $h_{ij}$s are bounded, which implies that the right-hand side in \eqref{eq:last_approx} is of order $\eps$.
\qedhere
\end{proof}

\section{Stability of entropic transport plans}\label{app: stability}

\begin{proof}[Proof of \Cref{prop: entmap_stability}]
Note that we may assume without loss of generality that $\nu\ll \nu'$ and that $\kl{\nu}{\nu'}<\infty$, for otherwise the bound is vacuous. For notational convenience, we omit the dependence on $\eps$ in the subscripts.
Write $\pi^{\mu, \nu} = \gamma^{\mu, \nu}(x, y) \rd \mu(x) \rd \nu(y)$ for the entropic optimal plan between $\mu$ and $\nu$, where $\gamma^{\mu, \nu} = \exp\left(\frac 1 \eps(\langle x, y \rangle - \phi^{\mu \to \nu}(x) - \psi^{\mu \to \nu}(y))\right)$, and analogously define $\gamma^{\mu', \nu'} =  \exp\left(\frac 1 \eps(\langle x, y \rangle - \phi^{\mu' \to \nu'}(x) - \psi^{\mu' \to \nu'}(y))\right)$.

Consider the measure $\gamma^{\mu', \nu'}(x, y) \dd \mu(x) \dd \nu'(y)$.
The first-order optimality condition for $(\phi^{\mu' \to \nu'},  \psi^{\mu' \to \nu'})$ implies that
\begin{equation}
\int \gamma^{\mu', \nu'}(y) \dd \nu'(y) = 1 \quad \forall x \in \Omega\,,
\end{equation}
so that $\gamma^{\mu', \nu'}(x, y) \dd\nu'(y)$ is a probability measure.
Let us write $\rd \pi^x(y) = \gamma^{\mu, \nu}(x, y) \dd\nu(y)$ and $\rd \rho^x(y) = \gamma^{\mu', \nu'}(x, y) \dd\nu'(y)$.

We make the following observations: first, $T^{\mu\to \nu}(x) = \int y \dd \pi^x(y)$ and $T^{\mu'\to \nu'}(x) = \int y \dd \rho^x(y)$.
Second, the support of $\rho^x$ lies inside $B(0; R)$; since any Lipschitz function $f$ on $B(0; R)$ satisfies $\sup_x f(x) - \inf_x f(x) \leq 2R$, Hoeffding's lemma~\citep[see][Lemma 2.2]{BouLugMas13} implies that if $f$ is Lipschitz and $\int f \dd \rho^x = 0$, then
\begin{equation*}
\int e^{t f} \dd \rho^x \leq e^{2 R^2 t^2} \quad \forall t \in \RR\,.
\end{equation*}
This implies \citep[Theorem 3.1]{BobGot99} that
\begin{equation}
W_1(\pi^x, \rho^x)^2 \leq 8 R^2 \kl{\pi^x}{\rho^x}\,.
\end{equation}
Third, Jensen's inequality implies that for any coupling $\gamma$ between $\pi^x$ and $\rho^x$,
\begin{equation}
\int \|y - y'\| \dd \gamma(y, y') \geq \left\|\int (y - y') \dd \gamma(y, y') \right\| = \|T^{\mu\to\nu}(x) - T^{\mu'\to \nu'}(x)\|\,,
\end{equation}
so that in particular, $\|T^{\mu\to\nu}(x) - T^{\mu'\to\nu'}(x)\| \leq W_1(\pi^x, \rho^x)$.
Combining these facts, we obtain
\begin{equation}
\frac{1}{8R^2} \|T^{\mu\to \nu}(x) - T^{\mu'\to \nu'}(x)\|^2 \leq \kl{\pi^x}{\rho^x} = \int \log \left(\frac{\gamma^{\mu, \nu}}{\gamma^{\mu', \nu'}}(x, y)\frac{\rd \nu}{\rd \nu'}(y)\right)\gamma^{\mu, \nu}(x, y) \dd\nu(y) \,.
\end{equation}
Integrating both sides of this equation with respect to $\mu$ yields
\begin{equation}
\frac{1}{8R^2} \|T^{\mu\to \nu}(x) - T^{\mu'\to \nu'}(x)\|^2_{L^2(\mu)} \leq \int \log \left(\frac{\gamma^{\mu, \nu}}{\gamma^{\mu', \nu'}}(x, y)\frac{\rd \nu}{\rd \nu'}(y)\right) \dd \pi^{\mu, \nu}(x, y)\,.
\end{equation}
Expanding the definition of $\gamma^{\mu, \nu}$ and $\gamma^{\mu', \nu'}$ and using that $$\int \log \frac{\rd \nu}{\rd \nu'}(y) \dd \pi^{\mu, \nu}(x, y) = \int \log \frac{\rd \nu}{\rd \nu'}(y) \dd \nu(y) = \kl{\nu}{\nu'}$$ yields the claim. 
\end{proof}
	We now record two corollaries of this bound, which apply when either the source or the target measures of the entropic maps agree.
\begin{corollary}\label{cor:same_mu}
	For any $\mu, \nu, \nu'$ supported in $B(0; R)$,
	\begin{equation}
	\frac{1}{8 R^2} \|T^{\mu \to \nu}_\ep - T^{\mu \to \nu'}_\ep\|_{L^2(\mu)}^2 \leq \ep^{-1} \int (\psi_\eps^{\mu \to \nu'} - \psi_\eps^{\mu \to \nu}) \dd(\nu - \nu') + \kl{\nu}{\nu'}
	\end{equation}
\end{corollary}
\begin{proof}
We apply \cref{prop: entmap_stability} with $\mu = \mu'$, which yields (once again omitting the dependency in $\eps$)
\begin{equation}\label{same_mu}
\frac{1}{8 R^2} \|T^{\mu \to \nu}_\ep - T^{\mu \to \nu'}_\ep\|_{L^2(\mu)}^2 \leq \ep^{-1} \left(\int (\varphi^{\mu \to \nu'} - \varphi^{\mu \to \nu}) \dd \mu + \int (\psi^{\mu \to \nu'} - \psi^{\mu \to \nu}) \dd \nu\right) + \kl{\nu}{\nu'}\,.
\end{equation}
By definition, $(\varphi^{\mu\to \nu'}, \psi^{\mu\to \nu'})$ minimizes the expression $$\int \varphi \dd \mu + \int \psi \dd \nu' + \ep \iint e^{(\langle x, y \rangle - \varphi(x) - \psi(y))/\ep} \dd \mu(x) \dd \nu'(y) - \ep$$, so, recalling that $\iint e^{(\langle x, y \rangle - \varphi^{\mu\to \nu'}(x) - \psi^{\mu\to \nu'}(y))/\ep} \dd \mu(x) \dd \nu'(y) = 1$, we have in particular
\begin{align*}
\int \varphi^{\mu\to \nu'} \dd \mu + \int \psi^{\mu\to \nu'} \dd \nu' & \leq \int \varphi^{\mu\to \nu} \dd \mu + \int \psi^{\mu\to \nu} \dd \nu' + \ep \iint e^{(\langle x, y \rangle - \varphi^{\mu\to \nu}(x) - \psi^{\mu\to \nu}(y))/\ep} \dd \mu(x) \dd \nu'(y) - \ep \\
& = \int \varphi^{\mu\to \nu} \dd \mu + \int \psi^{\mu\to \nu} \dd \nu'\,,
\end{align*}
where we have used that the first-order optimality condition for $(\varphi^{\mu\to \nu}, \psi^{\mu\to \nu})$ implies that $\iint e^{(\langle x, y \rangle - \varphi^{\mu\to \nu}(x) - \psi^{\mu\to \nu}(y))/\ep} \dd \mu(x) \dd \nu'(y) = 1$ as well (see \eqref{eq: opt_cond_1}).
This implies
\begin{equation}
\int (\varphi^{\mu\to \nu'} - \varphi^{\mu\to \nu}) \dd \mu \leq - \int (\psi^{\mu\to \nu'} - \psi^{\mu\to \nu}) \dd \nu'\,.
\end{equation}
	
Applying this inequality to~\eqref{same_mu} yields
	\begin{equation*}
	\frac{1}{8 R^2} \|T^{\mu \to \nu}_\ep - T^{\mu \to \nu'}_\ep\|_{L^2(\mu)}^2 \leq \ep^{-1} \int (\psi^{\mu\to \nu'} - \psi^{\mu\to \nu}) \dd(\nu - \nu') + \kl{\nu}{\nu'}. \qedhere
	\end{equation*}
\end{proof}
\begin{corollary}\label{cor:same_nu}
	For any $\mu, \mu', \nu$ supported in $B(0; R)$,
	\begin{equation}
	\frac{1}{8 R^2} \|T^{\mu \to \nu}_\ep - T^{\mu' \to \nu}_\ep\|_{L^2(\mu)}^2 \leq \ep^{-1}\int (\varphi_\eps^{\mu'\to \nu} - \varphi_\eps^{\mu\to \nu}) \dd (\mu - \mu')\,.
	\end{equation}
\end{corollary}
\begin{proof}
	We apply \cref{prop: entmap_stability} with $\nu = \nu'$, yielding (dropping the dependency on $\eps$)
	\begin{equation}\label{same_nu}
	\frac{1}{8 R^2} \|T^{\mu \to \nu} - T^{\mu' \to \nu}\|_{L^2(\mu)}^2 \leq \ep^{-1} \left(\int (\varphi^{\mu'\to \nu} - \varphi^{\mu\to \nu}) \dd \mu + \int (\psi^{\mu'\to \nu} - \psi^{\mu\to \nu}) \dd \nu\right)\,.
	\end{equation}
	An argument analogous to the one used in the proof of \cref{cor:same_mu} gives the inequality
	\begin{equation}
	\int \varphi^{\mu'\to \nu} \dd \mu' + \int \psi^{\mu'\to \nu} \dd \nu \leq \int \varphi^{\mu\to \nu} \dd \mu' + \int \psi^{\mu\to \nu} \dd \nu\,,
	\end{equation}
	or, equivalently,
	\begin{equation}
	\int (\psi^{\mu'\to \nu} -\psi^{\mu\to \nu} ) \dd \nu \leq - \int (\varphi^{\mu'\to \nu} -\varphi^{\mu\to \nu}) \dd \mu'\,,
	\end{equation}
	and combining this inequality with~\eqref{same_nu} proves the claim.
\end{proof}

\section{Strong convexity  of the entropic semi-dual problem}

\begin{proposition}[Strong convexity of $F^{\mu\to\nu}_\eps$]\label{prop:strong_convexity_F}
Let $\nu = \sum_{j=1}^J \nu_j \delta_{y_j}$ be a measure supported on $\{y_1,\dots,y_J\}\subseteq B(0;R)$ and let $\mu$ supported on a compact convex set $\Omega \subseteq B(0;R)$ with a density $p$ satisfying $p_{\min}\leq p \leq p_{\max}$ for some $p_{\max}\geq p_{\min}>0$. 
   For  $\psi\in \R^J$, define $\nu_\psi = \int \pi_\eps^x[\psi] \dd \mu(x)$ and assume that $\nu_\psi\geq \lambda \nu$ for some $0<\lambda \leq 1$. Then, we have for $\eps\in(0,1)$
    \begin{equation}
        F^{\mu\to \nu}_\eps(\psi) - \min_\psi F^{\mu\to \nu}_\eps \geq C \lambda  \cdot \mathrm{Var}_\nu(\psi-\psi^{\mu\to\nu}_\eps),
    \end{equation}
    where $C =\left(e^{2R^2} \frac{p_{\max}}{p_{\min}} + \eps \right)^{-1} \frac{p_{\min}}{p_{\max}}$.
\end{proposition}

\begin{proof}
As $\mu$ and $\eps$ are fixed, we will simply write $\psi_\nu$ instead of $\psi_\eps^{\mu\to\nu}$, and write similarly $F_\nu = F_\eps^{\mu\to\nu}$.  Recall the definition \eqref{eq:shifted_potential} of the shifted potential $\tilde \psi_\nu(y_j) = \psi_\nu(y_j)-\eps \log \nu_j$. 
According to \citep[Theorem 3.2]{delalande2021nearly}, the functional $F_{\nu}$ is minimized at the vector $\tilde\psi_{ \nu}$, with 
   \begin{equation}
              \forall v\in \R^J,\quad  \mathrm{Var}_\nu(v) \leq \left(e^{2R^2} \frac{p_{\max}}{p_{\min}} + \eps \right) v^\top \nabla^2 F_\nu(\tilde \psi_{\nu}) v. 
    \end{equation}
    For $t\in [0,1]$, let $\psi_t = \tilde \psi_{ \nu} + t(\psi- \tilde \psi_{ \nu})$ and let $\nu_t = \int \pi_\eps^x[\psi_t]\dd \mu(x)$. The potential $\psi_t$ is the (shifted) entropic Brenier potential between $\mu$ and $\nu_t$, so that it minimizes the functional $F_{ \nu_t}$ (see \Cref{sec:reminder}). Also, note that $\nabla^2 F_{\nu}$ does not depend on $\nu$, so that 
    \begin{equation}
           v^\top \nabla^2 F_{\nu}(\psi_t) v=  v^\top \nabla^2 F_{\nu_t}(\psi_t) v \geq \left(e^{2R^2} \frac{p_{\max}}{p_{\min}} + \eps \right)^{-1}  \mathrm{Var}_{\nu_t}(v).
    \end{equation}
    Let $v = \psi-\psi^{\mu\to\nu}_\eps$. 
    A Taylor expansion of $F_{\nu}$ gives
    \begin{equation}
        F_{\nu}(\psi)- F_{ \nu}(\tilde \psi_{\nu}) = \int_0^1 v^\top \nabla^2 F_{\nu}(\psi_t) v \dd t \geq \left(e^{2R^2} \frac{p_{\max}}{p_{\min}} + \eps \right)^{-1}  \int_0^1 \mathrm{Var}_{\nu_t}(v) \dd t.
    \end{equation}

\begin{lemma}\label{lem:prekopa}
   Write $\nu_t = \sum_{j=1}^J \nu_{t,j} \delta_{y_j}$. Then, for all $t\in [0,1]$ and $j\in[J]$, we have $\nu_{t,j} \geq \frac{p_{\min}}{p_{\max}} \nu_{0,j}^{1-t}\nu_{1,j}^t$.
\end{lemma}
This lemma is enough to conclude the proof. Indeed, $\nu_1 = \nu_\psi\geq \lambda\nu$, so that it implies that $\mathrm{Var}_{\nu_t}(v) \geq \frac{p_{\min}}{p_{\max}} \lambda \mathrm{Var}_\nu(v)$.
\end{proof}

\begin{proof}[Proof of \Cref{lem:prekopa}]
According to \citep[Proof of Proposition 4.1]{delalande2021nearly},
\begin{equation}
    \Phi_\eps(\psi_t)(tx+(1-t)y) \leq t \Phi_\eps(\tilde\psi^{\mu\to\nu}_\eps)(x)+ (1-t)\Phi_\eps(\psi)(y).
\end{equation}
Therefore, if we let $h_t(x) = e^{(\langle x,y_j \rangle -\psi_{t}(y_j)- \Phi_\eps(\psi_t)(x))/\eps}$, then we have $h_t(tx + (1-t)y) \geq h_0(x)^t h_1(y)^{1-t}$. By the Prékopa-Leindler inequality,
     \begin{align*}
         \nu_{t,j} = \int h_t(x) \dd \mu(x) \geq p_{\min} \int_{\cX} h_t(x) \dd x \geq p_{\min} \left( \int_{\cX} h_0(x) \dd x\right)^t\left( \int_{\cX}h_1(x) \dd x\right)^{1-t} \geq \frac{p_{\min}}{p_{\max}}  \nu_{0,j}^{t}\nu_{1,j}^{1-t}. 
     \end{align*} 
\end{proof}

\begin{proof}[Proof of \Cref{prop:stability_potentials}]
As in the previous proof, we drop the $\eps$ and $\mu$ dependency in our notation. Write $\nu_k = \sum_{j=1}^J \nu_{k,j}\delta_{y_j}$ for $k=0,1$, and define as before the shifted potentials $\tilde \psi_{\nu_k}(y_j) = \psi_{\nu_1}(y_j)-\eps \log \nu_{k,j}$. Let $\theta>0$ be a parameter to fix.    According to \Cref{prop:strong_convexity_F}, \Cref{lem:young}, and using the inequality $F_{\nu_1}(\tilde\psi_{\nu_1}) \leq F_{\nu_1}(\tilde\psi_{\nu_0})$, we have 
    \begin{align*}
        C\lambda \mathrm{Var}_{\nu_0}(\tilde \psi_{\nu_1}-\tilde\psi_{\nu_0}) \leq  F_{\nu_0}(\tilde\psi_{\nu_1})-F_{\nu_0}(\tilde\psi_{\nu_0}) &\leq F_{\nu_0}(\tilde\psi_{\nu_1})-F_{\nu_1}(\tilde\psi_{\nu_1}) + F_{\nu_1}(\tilde\psi_{\nu_0}) - F_{\nu_0}(\tilde\psi_{\nu_0}) \\
        &= \int (\tilde \psi_{\nu_1}-\tilde \psi_{\nu_0}) (\dd \nu_0-\dd \nu_1) \\
        &\leq  \frac{\theta}{2}  \mathrm{Var}_{\nu_0}(\tilde \psi_{\nu_1}-\tilde \psi_{\nu_0})+ \frac{1}{2\theta}\chis{\nu_1}{\nu_0}.
    \end{align*}
    We pick $\theta = C\lambda$ to conclude that
    \begin{equation}
        \mathrm{Var}_{\nu_0}(\tilde \psi_{\nu_1}-\tilde\psi_{\nu_0}) \leq \frac{1}{(C\lambda)^2} \chis{\nu_1}{\nu_0}.
    \end{equation}
Therefore, using the inequality $|\log(a/b)|\leq |a-b|/\min\{a,b\}$ for $a,b>0$,
\begin{align*}
       \mathrm{Var}_{\nu_0}(\psi_{1}-\psi_{0}) &\leq 2 \mathrm{Var}_{\nu_0}(\tilde \psi_{1}-\tilde \psi_{0}) + 2 \sum_{j= 1}^J \nu_{0,j} \left( \log\left(\frac{ \nu_{1,j}}{\nu_{0,j}}\right) \right)^2 \\
       &\leq \frac{2}{(C\lambda)^2}  \chis{\nu_1}{\nu_0}  + 2 \sum_{j=1}^J \nu_{0,j} \left(\frac{\nu_{1,j}-\nu_{0,j}}{\min\{\nu_{0,j},\nu_{1,j}\}}\right)^2  \\
       &\leq \frac{2}{(C\lambda)^2}  \chis{\nu_1}{\nu_0}  + \frac{2}{\lambda^2} \sum_{j=1}^J \frac{1}{\nu_{0,j} }(\nu_{1,j}-\nu_{0,j})^2  \leq \left(\frac{2}{(C\lambda)^2}  + \frac{2}{\lambda^2}\right)\chis{\nu_1}{\nu_0}. \qedhere
\end{align*}

\end{proof}

\section{Control of the fluctuations in the one-sample case}\label{sec: proofs in semidiscrete case}

\begin{lemma}[Sample complexity in the one-sample case]\label{lem:one_sample}
Assume that $P$ satisfy \textbf{(A)} and that $Q$ satisfy \textbf{(B)}. Then, it holds that $\E\| T_{\eps}^{P\to Q_n}-T_\eps\|^2_{L^2(P)} \lesssim \eps^{-1}n^{-1}$.
\end{lemma}
\begin{proof}
To ease notation, we write $T_{\eps,n} =T_{\eps}^{P\to Q_n}$ and $\psi_{\eps,n} = \psi_\eps^{P\to Q_n}$. 
    As explained in \Cref{sec:main_results}, the stability result \Cref{prop: entmap_stability} implies that
    \begin{equation}
        \E\| T_{\eps,n}-T_\eps\|^2_{L^2(P)} \leq  \frac{8R^2}{\eps}\Big( \frac{\E [\mathrm{Var}_Q(\psi_{\eps,n} - \psi_\eps)]}{2} + \frac{\E [\chis{Q_n}{Q}]}{2} \Big)  + 8R^2 \E [\chis{Q_n}{Q}]\,.
    \end{equation}
    Write $Q= \sum_{j=1}^J q_j \delta_{y_j}$ and $Q_n = \sum_{j=1}^J \hat q_j \delta_{y_j}$, and introduce the event $E=\{\forall j \in [J],\ \hat q_j\geq q_j/2\}$. If $E$ is satisfied, we have $Q_n \geq Q/2$, so that  \Cref{prop:stability_potentials} yields
    \begin{equation}
        \mathrm{Var}_Q(\psi_{\eps,n} - \psi_\eps) \leq C \chis{Q_n}{Q}.
    \end{equation}
    If $E$ is not satisfied, we use the fact that the entropic potentials have a bounded amplitude (see \Cref{lem:bounded_amplitude}), to obtain that
    \begin{equation}
        \mathrm{Var}_Q(\psi_{\eps,n} - \psi_\eps)\leq C'.
    \end{equation}

    \begin{lemma}\label{lem:Q_n_is_not_small}
        Let $E$ be the event that $Q_n\geq Q/2$. Then $\PP(E^c) \leq  Je^{-cq_{\min} n}$ for some $c>0$.
    \end{lemma}
    \begin{proof}
        By \citep[Exercise 2.3.2]{vershynin2018high}, we have $\PP(E^c) \leq \sum_{j=1}^J \PP(\hat q_j < q_j/2) \leq Je^{-cq_{\min} n}$ for some $c>0$.
    \end{proof}
     We obtain
    \begin{equation}
        \E\|\hat{T}_{\eps,n} - T_\eps\|^2_{L^2(P)} \lesssim \frac{R^2}{\eps} \E[\chis{Q_n}{Q}] + \frac{R^2}{\eps}Je^{-cq_{\min} n} \lesssim \eps^{-1}n^{-1}
    \end{equation}
    by \Cref{lem: kl_expectation_bound}.
\end{proof}

\section{Control of the fluctuations in the two-sample case}\label{sec: proofs_two_sample}

The goal of this section is to prove \Cref{thm: entmap_semidiscrete_samp_full}. We will actually prove a more general result, and show that \emph{for any discrete measure $\nu =\sum_{j=1}^J \nu_j\delta_{y_j}$} supported on $\{y_1,\dots,y_J\}$ with $\nu_j\geq \nu_{\min}>0$ for all $j\in [J]$, we have for  $\log(1/\eps) \lesssim n/\log(n)$,
\begin{equation}\label{eq:two_sample_general}
    \E \| T_{\eps}^{P_n\to \nu}- T_\eps^{P\to \nu}\|_{L_2(P)}^2 \lesssim \eps^{-1}n^{-1}.
\end{equation}
\Cref{thm: entmap_semidiscrete_samp_full} follows from \eqref{eq:two_sample_general} by conditioning on $Q_n$. Let $E$  be the event that $Q_n\geq Q/2$. Then, by \Cref{lem:Q_n_is_not_small},
\begin{align*}
    \E \|\hat T_\eps - T_\eps^{P\to Q_n}\|_{L_2(P)}^2 &\leq \E\left[  \E[ \|\hat T_\eps - T_\eps^{P\to Q_n}\|_{L_2(P)}^2| Q_n] \1\{E\} \right] + R^2 \PP(E^c) \\
    &\leq C\eps^{-1}n^{-1} + R^2 Je^{-cq_{\min} n} \lesssim \eps^{-1}n^{-1}.
\end{align*}
We obtain \Cref{thm: entmap_semidiscrete_samp_full} by combining this bound with \Cref{lem:one_sample}. 

To prove \eqref{eq:two_sample_general}, we first use \Cref{cor:same_nu} which yields
\begin{equation}\label{eq:first_step_2sample}
\begin{split}
        \E \| T_{\eps}^{P_n\to \nu}- T_\eps^{P\to \nu}\|_{L_2(P)}^2& \leq 8R^2 \eps^{-1} \E \int (\phi_\eps^{P_n\to \nu}-\phi_\eps^{P\to \nu})\dd (P_n-P)\\
        &=8R^2 \eps^{-1}\E \int (\Phi_\eps(\tilde\psi_\eps^{P_n\to \nu})-\Phi_\eps(\tilde \psi_\eps^{P\to \nu}))\dd (P_n-P),
        \end{split}
\end{equation}
where we recall that for a potential $\psi$, the shifted potential $\tilde\psi$ is given by $\tilde\psi_j=\psi_j-\eps\log \nu_j$. 
The remainder of the proof consists in bounding this integral by using localization arguments and standard bounds on suprema of empirical processes.
Our first goal is to show that the potential $\psi_\eps^{P_n\to \nu}$ is close to  to the potential $\psi_\eps^{P\to \nu}$ for the $\infty$-norm. It will be convenient to work with the ``$L_\infty$-variance''
\begin{equation}
    \mathrm{Var}_\infty(\psi) = \inf_{c\in \R} \max_{j\in [J]} |\psi(y_j)-c|^2 =  \left(\frac{\max \psi-\min \psi}{2} \right)^2.
\end{equation}
As the measure $\nu$ is lower bounded, it holds that
\begin{equation}
    \mathrm{Var}_\nu(\psi) \geq \nu_{\min} \mathrm{Var}_\infty(\psi).
\end{equation}

\begin{lemma}[Supremum of $\eps$-Legendre transforms]\label{lem: expectation_suprema}
    Let $\psi_0$ be a fixed potential and let $\tau>0$. Then, for all $j\in[J]$,
    \begin{align}
    &\E \left[ \sup_{\mathrm{Var}_\infty(\psi-\psi_0) \leq \tau^2} \left|\int (\pi_\eps^x(\psi)_j-\pi_\eps^x(\psi_0)_j)\dd( P- P_n)(x) \right| \right] \leq C \sqrt{\frac{J \max\{\log(\tau/\eps),1\}}{n}} \\
        &\E \left[ \sup_{\mathrm{Var}_\infty(\psi-\psi_0) \leq \tau^2}\left| \int (\Phi_\eps(\psi)(x)-\Phi_\eps(\psi_0))(x)\dd( P- P_n)(x)\right| \right] \leq C\tau \sqrt{ \frac{J}{n}}
    \end{align} 
    for some absolute constant $C$.
\end{lemma}
\begin{proof}
For a metric space $(A,d)$ and $u>0$, we let $N(u,A,d)$ be the covering number of $A$ at scale $u$, that is the smallest number of balls of radius $u$ needed to cover $A$. 
Let $B$ be the $L_\infty$-ball of radius $\tau$ in $\R^J$, centered at $\psi_0$, and let $\|\cdot\|_\infty$ denote the $\infty$-norm.
For $0<u\leq \tau$, we have $\log N(u,B, \|\cdot\|_\infty) \leq J\log(\tau/u)$.

We start with the second inequality. Note that $\psi\mapsto \Phi_\eps(\psi)$ is $1$-Lipschitz continuous, and that the functional $\Phi_\eps$ satisfies $\Phi_\eps(\psi+c)= \Phi_\eps(\psi)+c$ for all $c\in \R$. Then the set $\{\psi:\ \mathrm{Var}_\infty(\psi-\psi_0)\leq \tau^2\}$ is equal to the set $\{\psi +c:\ \psi\in B,\ c\in \R\}$. As $\int c\dd (P- P_n) =0$, we can therefore restrict the supremum to vectors $\psi \in B$. Furthermore, an envelope function of the class $\{\Phi_\eps(\psi)-\Phi_\eps(\psi_0):\ \psi \in B\}$ is the constant function equal to $\tau$. Therefore, by  \Cref{lem: suprema_empirical_process}, we obtain
\begin{equation}
    \begin{split}
     \E \left[ \sup_{\|\psi-\psi_0\|_\infty \leq \tau}\left| \int (\Phi_\eps(\psi)-\Phi_\eps(\psi_0))(\dd P-\dd P_n) \right|\right] &\leq \frac{c_0}{\sqrt{n}} \int_0^{c_1\tau} \sqrt{J\log 2 N(u,\{\Phi_\eps(\psi):\ \psi\in B\}, \|\cdot\|_\infty )} \dd u \\
     &\leq  \sqrt{\frac{c_3J  \tau}{n}}. \qedhere
     \end{split}
\end{equation}

We repeat the same argument for the first inequality. The functional $\pi_\eps^x$ is invariant by translation: $\pi_\eps^x(\psi+c)= \pi_\eps^x(\psi)$ for all $c\in \R$. This implies that
\begin{align*}
    \sup_{\mathrm{Var}_\infty(\psi-\psi_0) \leq \tau^2}\left| \int (\Phi_\eps(\psi)(x)-\Phi_\eps(\psi_0))(x)\dd( P- P_n)(x)\right| =\\ 
\sup_{\|\psi-\psi_0\|_\infty \leq \tau}\left| \int (\Phi_\eps(\psi)(x)-\Phi_\eps(\psi_0))(x)\dd( P- P_n)(x)\right|.
\end{align*}

As the function $\psi\mapsto \pi_\eps^x(\psi)_j$ is $\eps^{-1}$-Lipschitz continuous for every $x\in \R^d$, we have for $0< u\leq \tau/\eps$,
\[\log N(u,\{x\mapsto \pi_\eps^x(\psi)_j:\ \psi\in B\}, \|\cdot\|_\infty )\leq J\log(\tau /(u\eps))\,.\]
Remarking furthermore that $0\leq \pi_\eps^x(\psi)_j\leq 1$ (so that the class of functions $\{x\mapsto \pi_\eps^x(\psi)_j:\ \psi\in B\}$ admits the constant function $1$ as an envelope function), we obtain the following control using \Cref{lem: suprema_empirical_process}:
\begin{align*}
      &\E \left[ \sup_{\|\psi-\psi_0\|_\infty \leq \tau} \left|\int (\pi_\eps^x(\psi)_j-\pi_\eps^x(\psi_0)_j)(\dd P-\dd P_n)(x)\right| \right]\hspace{-.05cm}  \\
      &\qquad \leq \hspace{-.05cm} \frac{c_0}{\sqrt{n}} \int_0^{c_1} \hspace{-.2cm}\sqrt{J\log 2 N(u,\{x\mapsto \pi_\eps^x(\psi)_j:\ \psi\in B\}, \|\cdot\|_\infty )} \dd u \\
     &\qquad \leq  \sqrt{\frac{c_2J  \max\{\log(\tau/\eps),1\}}{n}},
\end{align*}
  where $c_0$, $c_1$ and $c_2$ are absolute constants, and the last line follows from arguing  whether $c_1<\tau /\eps$ or not. 
\end{proof}

\begin{proposition}\label{prop:bound_potentials_two_sample}
Assume that $P$ satisfies \textbf{(A)} and let $\nu = \sum_{j=1}^J \nu_j \delta_{y_j}$ be a measure supported on $\{y_1,\dots,y_J\} \subset B(0;R)$, with $\nu_j \geq q_{\min}$ for all $j\in [J]$. Then, for all $0<\eps\leq 1$ with $\log(1/\eps) \lesssim n/\log(n)$, it holds that 
\begin{equation}
\E \mathrm{Var}_\infty(\tilde \psi_\eps^{P_n\to \nu} - \tilde\psi_\eps^{P\to \nu}) \lesssim n^{-1}. 
\end{equation}
\end{proposition}

\begin{proof}
To alleviate notation, we will write $\psi_n = \psi_\eps^{P_n\to \nu}$ and $\psi_0 = \psi_\eps^{P\to \nu}$. Similarly, we write $F_n =F^{P_n\to \nu}_\eps $ and  $F_0=F^{P\to \nu}_\eps$. 
    Let $\nu_n = \int \pi_\eps^x(\psi_\eps^{P_n\to \nu}) \dd P(x)$. Under the event $E=\{\nu_n\geq \nu/2\}$, we have according to \Cref{prop:strong_convexity_F} and the fact that $\tilde\psi_n$ minimizes $F_n$, 
    \begin{equation}\label{eq:strong_convexity_two_sample}
    \begin{split}
C\nu_{\min} \mathrm{Var}_\infty(\tilde\psi_n - \tilde\psi_0)  &\leq C \mathrm{Var}_\nu(\tilde\psi_n -\tilde \psi_0) \\
&\leq F_0(\tilde\psi_n) - F_0(\tilde\psi_0) \\
&\leq F_0 (\tilde\psi_n) - F_n (\tilde\psi_n) + F_n(\tilde\psi_0) - F_0(\tilde\psi_0) \\
        &= \int (\Phi_\eps(\tilde\psi_n)- \Phi_\eps(\tilde\psi_0)) \dd (P-P_n)  
    \end{split}        
    \end{equation}
Let us bound $\PP(E^c)$. As $\tilde\psi_n$ is the minimum of $F_n$, we have $\nu=\int \pi_\eps^x(\tilde\psi_n)_j \dd P_n(x)$ (see \Cref{sec:reminder}). Therefore, we may write $\nu_{n,j} =  \int \pi_\eps^x(\tilde\psi_n)_j \dd P_n(x) +  \int \pi_\eps^x(\tilde\psi_n)_j \dd (P-P_n)(x) = \nu_{j}  +  Z_j$, 
where 
\[ Z_j =  \int \pi_\eps^x(\tilde\psi_n)_j \dd (P-P_n)(x) =  \int (\pi_\eps^x(\tilde\psi_n)_j -\pi_\eps^x(\tilde\psi_0)_j)\dd (P-P_n)(x).\]
Note that $\mathrm{Var}_\infty(\tilde\psi_n-\tilde\psi_0) \lesssim R^2$ (see \Cref{lem:bounded_amplitude}), so that by \Cref{lem: expectation_suprema} and \Cref{lem: suprema_empirical_process},
\begin{equation}\label{eq: bound_PEc}
   \PP(E^c) \leq \sum_{j=1}^J \PP( |Z_j|> \nu_j/2)  \leq J\exp\left(-c \frac{\sqrt{n}q_{\min}}{ (\sqrt{J\log(1/\eps)}+ \log n} \right)\lesssim n^{-1},
\end{equation}
under the condition $\log(1/\eps) \lesssim n/\log(n)$. 

For $k\geq 0$, let $a_k = 2^k/\sqrt{n}$ and fix some $p>2$. Let $$B_{a} =\sup_{\mathrm{Var}_\infty(\psi-\tilde\psi_0) \leq a^2}\left| \int(\Phi_\eps(\psi)- \Phi_\eps(\tilde\psi_0)) \dd (P-P_n) \right|$$. Assume that $E$ is satisfied and that $\mathrm{Var}_\infty(\tilde\psi_0-\tilde\psi_n) \in [a^2,b^2]$. Then, according to \eqref{eq:strong_convexity_two_sample}, it holds that $B_b \geq ca^2$. Using Markov's inequality, \Cref{lem: expectation_suprema} and \Cref{lem: suprema_empirical_process}, we bound
\begin{align*}
    \E \mathrm{Var}_\infty(\tilde\psi_n-\tilde\psi_0) &\leq a_0^2 + \sum_{k\geq 0} \PP(\mathrm{Var}_\infty(\tilde\psi_n-\tilde\psi_0) \in [a_k^2,a_{k+1}^2]\text{ and } E) a_{k+1}^2 + C\PP(E^c) \\
    &\lesssim n^{-1} +
    \sum_{k\geq 0} \PP\left(B_{a_{k+1}} \geq ca_k^2 \right) a_{k+1}^2 \lesssim n^{-1} + \sum_{k\geq 0} \frac{\E[ B_{a_{k+1}}^p]}{a_k^{2p}} a_{k+1}^2  \\
    &\lesssim n^{-1} +\sum_{k\geq 0} \frac{(2^k/n)^p}{(4^k/n)^p} \frac{4^{k+1}}{n} \lesssim  n^{-1} + \sum_{k\geq 0} \frac{2^{2k-pk}}{n}  \lesssim n^{-1}. \qedhere
\end{align*}
\end{proof}

\begin{proposition}
    Under the same assumptions than \Cref{prop:bound_potentials_two_sample}, it holds that
\begin{equation}
\E \| T_\eps^{P_n\to \nu} - T_\eps^{P\to \nu}\|^2_\infty \lesssim \eps^{-1}n^{-1}.
\end{equation}
\end{proposition}

\begin{proof}
    Let $Z= \mathrm{Var}_\infty(\tilde\psi_n-\tilde\psi_0)$. Let once again $a_k = 2^k/\sqrt{n}$ for $k\geq 1$, with $a_0=0$. Fix some $p>2$, with $q= \frac{p}{p-1}$. For $a>0$, let $D_a = \sup_{ \mathrm{Var}_\infty(\psi-\tilde\psi_0) \leq a^2} \left|\int ( \Phi_\eps(\psi)-\Phi_\eps(\tilde\psi_0))\dd (P-P_n )\right|$. By H\"older inequality and Markov inequality, we obtain,
    \begin{align*}
        &\E \int ( \Phi_\eps(\tilde\psi_n)-\Phi_\eps(\tilde\psi_0))\dd (P-P_n ) \\
        &\qquad \leq \sum_{k\geq 0} \E \left[ \1\{  Z\in [a_k^2,a_{k+1}^2]\} \sup_{ \mathrm{Var}_\infty(\psi-\tilde\psi_0) \leq a_{k+1}^2} \int ( \Phi_\eps(\psi)-\Phi_\eps(\tilde\psi_0))\dd (P-P_n ) \right] \\
        &\qquad \leq \E[D_{a_1}]+\sum_{k\geq 1} \left(\PP( Z\geq a_k^2) \right)^{1/q} \E\left[ D_{a_{k+1}}^{p}\right]^{1/p} \\
        &\qquad \lesssim n^{-1} + \sum_{k\geq 0} \left( \frac{\E[Z]}{a_k^2}\right)^{1/q}\frac{2^k}{n} \lesssim \sum_{k\geq 0}\frac{2^{k(1-2/q)}}{n} \lesssim n^{-1} ,
    \end{align*}
    where we use \Cref{prop:bound_potentials_two_sample}, \Cref{lem: expectation_suprema} and \Cref{lem: suprema_empirical_process} at the last line. \Cref{eq:first_step_2sample} then gives the conclusion.
\end{proof}

\section{A lower bound for the performance of the 1NN estimator}\label{1nn_lb}
In this section, we prove \cref{prop:1NN_suboptimal}.
We let $P$ be the Lebesgue measure on $\Omega = [0, 1]^d$, and let $y_0 = (0, 1/2, \dots, 1/2)$ and $y_1 = (1, 1/2, \dots, 1/2)$.
We denote by $P_n$ an empirical measure consisting of i.i.d.\ samples from $P$.
As in \cref{sec: proofs_two_sample}, we work in a general setting of a generic discrete target measure $\nu$, which may either be fixed or may be a random measure independent of $P_n$.
We let $\nu = \sum_{j=0, 1} \nu_j \delta_{y_j}$ for $\nu_0, \nu_1 \geq \frac 14$; this latter condition will hold with overwhelming probability if $\nu$ is an empirical measure $Q_n$ corresponding to $n$ i.i.d.\ samples from $Q = \frac 12 \delta_{y_0} + \frac 12 \delta_{y_1}$.
Following \citep{manole2021plugin}, we define the one-nearest neighbor estimator $\hat T_{\text{1NN}}$ in this general context by
\begin{equation*}
	\hat T_{\text{1NN}} (x) = \sum_{i=1}^n \sum_{j = {0, 1}} \bm 1_{V_i}(x) (n \hat \pi(X_i, y_j))\,,
\end{equation*}
where $\hat \pi$ is the empirical optimal coupling between $P_n$ and $\nu$.

We first examine the structure of the Brenier map $T_0=\nabla\phi_0$.
The considerations in \cref{sec: semidiscrete_background} imply that
\begin{equation*}
	T_0(x) = \begin{cases}
		y_0 & \langle e_1, x \rangle \leq \nu_0 \\
		y_1 & \langle e_1, x \rangle > \nu_0\,,
	\end{cases}
\end{equation*}
where $e_1$ is the first elementary basis vector.
The potential $\phi_0$ is not differentiable on the separating hyperplane $\langle e_1, x \rangle  = \nu_0$, which has measure $0$ under $P$, but we may arbitrarily assign points on this hyperplane to $y_0$.

Similar arguments imply that the empirical transport plan $\hat \pi$ between $P_n$ and $\nu$ has the following property: there exists a (random) threshold $\tau \in (0, 1)$ such that
\begin{equation*}
	\hat \pi(x, y_0) = \begin{cases}
		1 & \langle e_1, x \rangle < \tau \\
		0 & \langle e_1, x \rangle > \tau\,.
	\end{cases}
\end{equation*}
The set $\langle e_1, x \rangle = \tau$ may not have measure $0$ under $P_n$, and $\hat \pi(x, y_0)$ may take values strictly between $0$ and $1$ on this set.

The following lemma shows that $\tau$ is close to $\nu_0$ with high probability.
\begin{lemma}
	For any $t \geq 0$,
	\begin{equation*}
		\p{\tau \geq \nu_0 + t} \leq e^{-2 n t^2}\,.
	\end{equation*}
\end{lemma}
\begin{proof}
	If $\tau \geq \nu_0 + t$, this implies that $P_n(\{x: \langle e_1, x\rangle < \nu_0 + t\}) \leq \nu_0$.
	On the other hand, $n P_n(\{x: \langle e_1, x\rangle < \nu_0 + t\}$ is a $\mathrm{Bin}(n, \nu_0+t)$ random variable.
	The result then follows from Hoeffding's inequality~\citep[Theorem 2.8]{BouLugMas13}.
\end{proof}

Let us write $H$ for the halfspace $\{x: \langle e_1, x \rangle  \leq \nu_0\}$, and $\hat H$ for the halfspace $\{x: \langle e_1, x \rangle  \leq \tau\}$.
Let $x$ be any point in $\Omega$ such that $x\in H$.
We are interested in the event that there exists an element $X_i \in \{X_1, \dots, X_n\}$ such that a) $x \in V_i$ and b) $X_i \in \hat H^c$.
Call this event $\cE(x)$.
On this event, $\hat{T}_{\text{1NN}}(x) = y_1$ and $T_0(x) = y_0$, so $\|\hat{T}_{\text{1NN}}(x) - T_0(x)\|^2 = 1$.

We therefore obtain
\begin{align*}
	\E \|\hat{T}_{\text{1NN}} - T_0\|_{L^2(P)}^2 & = \E \int \|\hat{T}_{\text{1NN}}(x) - T_0(x)\|^2 \dd P(x) \\
	& \geq \E \int_H \|\hat{T}_{\text{1NN}}(x) - T_0(x)\|^2 \1\{\cE({x})\} \dd P(x) \\
	& \gtrsim \E \int_H \1\{\cE({x})\} \dd P(x) \\
	& = \int_H \p{\cE({x})} \dd P(x)\,,
\end{align*}
where the final equality follows from the Fubini--Tonelli theorem.

We now lower bound the probability of $\cE(x)$.
Let us write $\cA_t$ for the event that $\tau < \nu_0 + t$, for $t > 0$ to be specified, and write $H_t$ for the halfspace $\{x: \langle e_1, x \rangle  \leq \nu_0 + t\}$.
Given any $x \in H$, write $\Delta = d(x, H_t^c)$, and let $B$ be a ball of radius $2\Delta$ around $x$, intersected with $\Omega$.

Denote by $\cF(x)$ the event that  there are no samples in $V = B \cap H_t$ but there is at least one point in $B \cap H_t^c$.
Then $\cF(x) \cap \cA_t \subseteq \cE(x)$, since on $\cF(x)$ the nearest neighbor to $x$ must be a sample in $H_t^c$, and on $\cA_t$ we have $H_t^c \subseteq \hat H^c$.
\begin{lemma}
	\begin{equation*}
		\p{\cF(x) \cap \cA_t} \geq (1 - \operatorname{vol}(V))^n - (1 - \operatorname{vol}(B))^n - e^{-2n t^2}\,.
	\end{equation*}
\end{lemma}
\begin{proof}
	We first compute $\p{\cF(x)}$.
	The probability that there are no samples in $V$ is $(1 - \operatorname{vol}(V))^n$, and this event may be written as the disjoint union of $\cF(x)$ and the event that all of $B$ is empty.
	The latter event has probability $(1 - \operatorname{vol}(B))^n$.
	Therefore
	\begin{equation*}
		(1 - \operatorname{vol}(V))^n = \p{\cF(x)} + (1 - \operatorname{vol}(B))^n\,.
	\end{equation*}
	Since $\p{\cA_t^c} \leq e^{-2n t^2}$, the claim follows.
\end{proof}

We need the following lemma.
\begin{lemma}
	Assume that $\Delta > 0$ and that $d(x, \partial \Omega) \geq 2 \Delta$.
	There exist positive constants $c_{d,0} < 1$ and $c_{d,1}$ such that
	\begin{equation}
		\operatorname{vol}(V) \leq c_{d,0} \operatorname{vol}(B)
	\end{equation}
	and
	\begin{equation}
		\operatorname{vol}(B) \geq c_{d,1} \Delta^d
	\end{equation}
\end{lemma}

\begin{proof}
This is immediate from a scaling argument: since $d(x, \partial \Omega) \geq 2 \Delta$, the set $B$ is a Euclidean ball of radius $2 \Delta$, and the set $V$ is a Euclidean ball of radius $2 \Delta$ minus a spherical dome cut off by a hyperplane at distance $\Delta$ from the center.
When $\Delta = 1$, it is clear that the claimed inequalities hold, and the general case is obtained by dilation.
\end{proof}
We assume in what follows that $d(x, \partial \Omega) \geq 2 \Delta$.
The inequalities $(1 + x)^n \geq 1 + nx$ and $e^x \leq 1 + x + x^2$, valid for all $x \in [-1, 0]$ and $n \geq 1$, imply that for any $\delta > 0$ there exists a constant $c_{d, \delta} > 0$ such that if $\Delta \leq c_{d, \delta} n^{-1/d}$, then we will have
\begin{align}
	(1- \operatorname{vol}(V))^n & \geq 1 - n c_{d,0} \operatorname{vol}(B) \\
	(1 - \operatorname{vol}(B))^n & \leq e^{- n \operatorname{vol}(B)} \leq 1  - (1-\delta) n \operatorname{vol}(B)
\end{align}
Choosing $\delta$ sufficiently small, we obtain the existence of a small $c_{d,3} > 0$ such that if $\Delta \leq c_{d,3} n^{-1/d}$, then
\begin{equation*}
	(1 - \operatorname{vol}(V))^n - (1 - \operatorname{vol}(B))^n \geq C_d n \Delta^d\,.
\end{equation*}

Define $\Delta_n = c_{d,4} n^{-1/d}$.
Putting it all together, consider the set
\begin{equation*}
	S = \{x \in H \cap \Omega: \Delta_n/2 \leq d(x, H_t^c) \leq \Delta_n, d(x, \partial \Omega) \geq 2 \Delta_n\}\,.
\end{equation*}
The above considerations imply that $\p{\cE(x)} \geq C_d n (\Delta_n/2)^d - e^{-2n t^2} \geq C_d' - e^{-2n t^2}$ for all $x \in S$.
Choosing $t$ to be a sufficiently large constant multiple of $n^{-1/2}$, we obtain
\begin{equation*}
	\int_H \p{\cE({x})} \dd P(x) \geq \int_{S} \p{\cE({x})} \dd P(x) \gtrsim_d \operatorname{vol}(S)\,.
\end{equation*}
Since $t \asymp n^{-1/2}$, we will have that $t \ll \Delta_n$ for $n$ sufficiently large (as $d\geq 3$).
Therefore, for $n$ large enough, the set $S$ contains the set $$S' = \{x \in \Omega: \nu_0 - \Delta_n + t\leq \langle e_1, x \rangle \leq \nu_0 - \Delta_n/2 + t , 2 \Delta_n \leq \langle e_j, x \rangle \leq 1- 2 \Delta_n \quad \forall j = 2, \dots, d\}\,.$$
Since $\operatorname{vol}(S') \gtrsim_d \Delta_n \gtrsim n^{-1/d}$, the claim follows.

\section{Auxiliary lemmas}
\begin{lemma}[Young's inequality]\label{lem:young}
    Let $Q_0,Q_1$ be  probability measures with $Q_1\ll Q_0$ and let $f$ be a function. Then, for $\theta>0$,
    \begin{equation}
        \int f (\dd Q_0-\dd Q_1) \leq \frac{\theta \mathrm{Var}_{Q_0}(f)}{2} + \frac{\chis{Q_1}{Q_0}}{2\theta}.
    \end{equation}
\end{lemma}
\begin{proof}
Recall Young's inequality: for $a,b\in \R$, $ab\leq \frac{a^2}{2}+ \frac{b^2}{2}$. 
As the left-hand side is invariant by translation, we may assume without loss of generality that $\int f\dd Q_0 = 0$, so that $\mathrm{Var}_{Q_0}(f)=\int f^2 \dd Q_0$. 
    We write
    \begin{align*}
        \int f(\dd Q_0-\dd Q_1)&= \int (\sqrt{\theta} f)\frac{\left( 1-\frac{\dd Q_1}{\dd Q_0}\right) }{\sqrt{\theta}} \dd Q_0 \leq \frac \theta 2 \int f^2 \dd Q_0 + \frac 1{2\theta} \int \left( 1-\frac{\dd Q_1}{\dd Q_0}\right)^2 \dd Q_0 \\
        &= \frac{\theta \mathrm{Var}_{Q_0}(f)}{2} + \frac{\chis{Q_1}{Q_0}}{2\theta}.\qedhere
    \end{align*}
\end{proof}

\begin{lemma}[Expectation of empirical $\chi^2$-divergence]\label{lem: kl_expectation_bound}
Let $Q= \sum_{j=1}^J q_j \delta_{y_j}$ be a discrete measure supported on $J$ atoms, and let $Q_n$ denote its empirical measure, consisting of $n$ i.i.d.~samples. Then,
\begin{align}
 \E[ \chis{Q_n}{Q}] = \frac{J-1}{n}\,.
\end{align}
\end{lemma}
\begin{proof}
We can write $Q_n = \sum_{j=1}^J \hat{q}_j\delta_{y_j}$, where $\hat{q}_j$ is a binomial random variable with parameters $n$ and $q_j$.  
We obtain
\begin{align*}
 \chis{Q_n}{Q} = \sum_{j=1}^J \frac{(\hat{q}_j - q_j)^2}{q_j}\,.
\end{align*}
Taking expectations, our bound reads
\begin{align*}
\E[ \chis{Q_n}{Q} ]=\sum_{j=1}^J \frac{\text{Var}(\hat{q}_j)}{q_j} = \sum_{j=1}^J  \frac{ q_j(1-q_j)}{n q_j} = \frac{J-1}{n}. 
\end{align*}
 
\end{proof}

\begin{lemma}[Control of suprema of empirical processes]\label{lem: suprema_empirical_process}
Let $X_1,\dots,X_n$ be an i.i.d.~sample from some probability measure $P$ on $\R^d$, with $P_n$ the associated empirical measure. Consider $\cF$ a class of functions $\R^d\to \R$ with $\|f\|_\infty\leq A$ for all $f\in \cF$. For $u>0$, let $N(u)$ be the $u$-covering numbers of $\cF$, that is the minimal number of balls of radius $u$ for the $\|\cdot\|_\infty$-metric required to cover $\cF$. Then,
\begin{equation}
   \E \left[\sup_{f\in \cF} \left| \int f\dd (P_n-P) \right| \right] \leq \frac{C_0}{\sqrt{n}} \int_0^{C_1 A} \sqrt{\log 2N(u)} \dd u =: \frac{I}{\sqrt{n}}
\end{equation}
for two positive  absolute constants $C_0$ and $C_1$. Furthermore, for all $t>0$,
\begin{equation}
    \PP\left( \sup_{f\in \cF} \left| \int f\dd (P_n-P) \right| > t\right) \leq \exp\left( - \frac{C_2 \sqrt{n}t}{I+ A\log n} \right),
\end{equation}
for some  positive  absolute constant $C_2$.   Eventually, for all $p\geq 2$,
\begin{equation}
    \E\left[ \sup_{f\in \cF} \left| \int f\dd (P_n-P) \right|^p\right]^{1/p} \leq C_p \frac{I +  A}{\sqrt{n}}.
\end{equation}
\end{lemma}
\begin{proof}
    See \citep[Theorem 2.14.2 and Theorem 2.14.5]{vaart1996weak}.
\end{proof} 

\bibliography{ref}

\end{document}